\documentclass[11pt]{amsart}

\usepackage{amsfonts,amssymb,amsthm}
\usepackage{etoolbox}
\usepackage{latexsym}
\usepackage{mathrsfs}
\usepackage{bm,thmtools}

\usepackage{amsmath}
\usepackage{mathrsfs}

\usepackage{xcolor}
\usepackage{esint}
\usepackage[pagebackref=true, colorlinks=true, citecolor=blue, pdfborder={0 0 .1},breaklinks]{hyperref}

\title[Nonlinear parabolic model equations]
{Existence and analyticity of solutions of nonlinear parabolic model equations with singular data}
\author[D.M. Ambrose]{David M. Ambrose}
\address{Drexel University, Department of Mathematics, Philadelphia, PA 19104, USA}
\email{dma68@drexel.edu}
\author[M.C. Lopes Filho]{Milton C. Lopes Filho}
\address{Instituto de Matematica, Universidade Federal do Rio de Janeiro, Caixa Postal 68530,
Rio de Janeiro, RJ, 21941-909 Brazil}
\email{mlopes@im.ufrj.br}
\author[H.J. Nussenzveig Lopes]{Helena J. Nussenzveig Lopes}
\address{Instituto de Matematica, Universidade Federal do Rio de Janeiro, Caixa Postal 68530,
Rio de Janeiro, RJ, 21941-909 Brazil}
\email{hlopes@im.ufrj.br}



\begin{document}

\newtheorem{theorem}{Theorem}
\newtheorem{corollary}[theorem]{Corollary}
\newtheorem{lemma}[theorem]{Lemma}
\newtheorem{assumption}{Assumption}
\newtheorem{remark}[theorem]{Remark}
\newtheorem{claim}[theorem]{Claim}

\newcommand{\vertiii}[1]{{\left\vert\kern-0.25ex\left\vert\kern-0.25ex\left\vert #1
    \right\vert\kern-0.25ex\right\vert\kern-0.25ex\right\vert}}

\begin{abstract} We explore two approaches to proving existence and analyticity of solutions to
nonlinear parabolic differential equations.  One of these methods works well for more general nonlinearities,
while the second method gives stronger results when the nonlinearity is simpler.
The first approach uses the exponentially weighted
Wiener algebra, and is related to prior work of Duchon and Robert for vortex sheets.  The second approach uses
two norms, one with a supremum in time and one with an integral in time, with the integral norm
representing the parabolic gain of regularity.
As an example of the first approach we prove analyticity of small solutions of a class of generalized
one-dimensional Kuramoto-Sivashinsky  equations, which model the motion of flame fronts and other phenomena.
To illustrate the second approach, we prove existence and analyticity of solutions of
the dissipative Constantin-Lax-Majda equation (which models vortex stretching), with and without added advection, with two classes of rough data.
The classes of data treated include both data in the Wiener algebra with negative-power weights, as well
as data in pseudomeasure spaces with negative-power weights.
\end{abstract}

\maketitle

\section{Introduction}
For many partial differential equations arising from applications, diffusive effects are essential to the modeling.
For such equations,
the diffusion term allows gain of regularity over time, and this smoothing effect may allow for existence of smooth solutions
starting from even quite rough data.  We will demonstrate two approaches to bring out this behavior, showing existence
of solutions analytic at positive times starting from data in some cases so rough that the Fourier coefficients do not even
decay.  We illustrate the method on two different classes of parabolic model equations: a class of one-dimensional
generalized Kuramoto-Sivashinsky equations, and the Constantin-Lax-Majda equation (with and without advection).

The study of analyticity of solutions is highly relevant for numerical methods for these equations.  Of course, for spectral
methods, one must have exponential decay of the Fourier coefficients for the method to work.  
As observed in \cite{titi} for Ginzburg-Landau equations, Galerkin methods
converge exponentially fast, with the rate determined by the radius of analyticity of solutions.
Since we prove that our solutions
are analytic (and thus the Fourier coefficients do decay exponentially) at all positive times, it suggests that a method such as
a splitting method would be effective for computing solutions of these equations with rough data.  That is, to initialize the
numerical method, the parabolic semigroup could be applied to non-decaying Fourier data for a short time (such as for half of a timestep
if using Strang splitting, but precisely
how long depends on the particular splitting method), and this could be computed analytically.  This results in an exponentially
decaying Fourier series.  Then one applies the nonlinear evolution, which can be handled now by spectral methods.  The linear and
nonlinear evolutions can then be alternated numerically, and we expect that such a method could be proven to converge, as in
for example \cite{bealeMajdaViscousSplitting}.  

The first analytical method we use is a generalization of the method employed by Duchon and Robert to prove existence of solutions
for all time for the classical vortex sheet \cite{duchonRobert}.  In this method, one demonstrates that the mild formulation
of the evolution gives a contraction mapping in exponentially weighted Wiener algebras.  The first author and collaborators
have employed this strategy a number of times, including for mean field games \cite{CRAS2}, \cite{JMPA},
epitaxial growth of thin films \cite{BLMS}, ill-posed Boussinesq equations and classical vortex sheets
\cite{ambroseBonaMilgrom}, \cite{milgromAmbrose},
and the two-dimensional Kuramoto-Sivashinsky equation \cite{ambroseMazzucato}, \cite{ambroseMazzucato2}.  The
method has also been employed by Beck, Sosoe, and Wong to demonstrate existence of solutions for other problems
in interfacial fluid flow \cite{beckSosoeWong}.
A strength of the method is that by relying upon the algebra structure of the function spaces
(inherited from the Wiener algebra), analytic nonlinearities may be studied as easily as simpler (such as quadratic)
nonlinearities.  A drawback of the method is that because the algebra structure is needed, only a certain amount of
roughness in the data can be handled.  Specifically, the method works for data in the Wiener algebra, or with a
positive number of derivatives in the Wiener algebra.  The method does not allow the data to be the derivative of a
Wiener function.

We employ the first method in a study the family of one-dimensional generalized Kuramoto-Sivashinsky equations
\begin{equation}\label{ottoFamily}
u_{t}=u^{m}u_{x}-\nu\Lambda^{b}u + \Lambda^{a}u.
\end{equation}
The operator $\Lambda$ is $\Lambda=\sqrt{-\partial_{xx}},$ which has Fourier symbol $\hat\Lambda(k)=|k|.$
The parameters in \eqref{ottoFamily} satisfy $0<a<b,$ and thus the operator $\Lambda^{b}$ has the correct sign to
make this a forward-parabolic problem, although the lower-order term $\Lambda^{a}$ will typically have
a destabilizing influence.  The nonlinearity parameter $m$ is taken to be a positive integer.
When $a=4,$ $b=2,$ and $m=1,$ \eqref{ottoFamily} is exactly the one-dimensional Kuramoto-Sivashinsky equation.
For an exploration of more general dynamics for Kuramoto-Sivashinsky-like equations with more general nonlinearities,
we refer the reader to \cite{au}.
The analyticity of solutions of certain members of the family \eqref{ottoFamily} was
studied by Papageorgiou, Smyrlis, and Tomlin in
\cite{papageorgiouEtAl}.  In particular, the choice $a=b/2$ is made, and analyticity of the solutions of
the
resulting one-parameter family of equations (called there the \emph{Otto equations})
is studied both analytically and numerically.  The case $b=1$ is identified
as a borderline case for analyticity.  Naturally one would not expect analyticity for $b<1;$ in this case, one cannot
expect exponential decay of the Fourier series, although one could study the associated Gevrey regularity of solutions.
For $b>1,$ Papageorgiou et al. find that solutions after a long time settle down to be analytic with a certain radius
$\rho(b)>0,$ but that $\rho(b)\rightarrow0$ as $b\rightarrow 1+.$  Granero-Belinch\'{o}n and Hunter have also studied
this problem, giving a proof of analyticity of solutions for $b>1$ \cite{graneroBelinchonHunter}.  A related problem
has been studied by Kiselev, Nazarov, and Shterenberg, showing analyticity for $b\geq1$ but in the absence of the linearly
destabilizing term (i.e., the $\Lambda^{a}$ term is not present at all) \cite{kiselevEtAl}.
In contrast to the numerical result of Papageorgiou et al.,
we will show that small solutions starting from Wiener data instantly gain analyticity, even in the case $b=1.$

The second method we employ to demonstrate existence and analyticity of solutions of parabolic equations with rough
data is adapted from works such as that of Bae for the three-dimensional Navier-Stokes equations \cite{baePAMS}.
This built on the work of Bae, Biswas, and Tadmor, and uses two norms, one with a supremum in time and a low level
of spatial regularity, and the other with an integral in time and a higher level of spatial regularity
\cite{baeBiswasTadmor}.  This method is more restrictive in terms of nonlinearities which may be treated than the
first method based on the Duchon-Robert work \cite{duchonRobert}.  In particular, for this second method, we use
a fixed point theorem that relies upon a bilinear estimate; thus, we are restricted in applying this method to problems
with quadratic nonlinearity.  Fortunately, many problems of interest do have quadratic nonlinearities, such as the
Navier-Stokes equations and the Kuramoto-Sivashinsky equation \cite{ALN5}, \cite{ALN6}.
We now apply this method to certain generalizations of the Constantin-Lax-Majda equation
\cite{CLM}.  In particular, we treat the generalized Constantin-Lax-Majda equation which includes the effects of diffusion
and advection, with the form of advection being that introduced by Okamoto, Sakajo, and Wunsch \cite{okamotoEtAl}.
We prove several results for the generalized Constantin-Lax-Majda equations, including existence of solutions for
initial data in negative-index Wiener-type spaces, and existence of solutions in negative-index pseudomeasure
spaces.  We believe that this is the first work exploring the solvability of generalizations of the Constantin-Lax-Majda
equation with such low regularity data.

Of course other authors have used many times similar two-norm approaches; however, our method differs from the usual
Kato two-norm method \cite{kato}.  In the Kato approach, when studying an unknown, $v,$
data is taken in a Banach space $\mathscr{X}$
and solutions are shown to also be in a space $\mathscr{Y}$ at positive times, through control of
\begin{equation}\label{katoVersion}
\sup_{t>0} t^{\rho}\|v(\cdot,t)\|_{\mathscr{Y}},
\end{equation}
for some appropriately chosen value of $\rho>0.$  This is quite prescriptive of the behavior of the $\mathscr{Y}$-norm
as $t\rightarrow0^{+}.$ Following Bae \cite{baePAMS}, we instead control the integral with respect to
time of the higher-regularity norm,
\begin{equation}\label{baeVersion}
\int_{0}^{\infty}\|v(\cdot,t)\|_{\mathscr{Y}}\ \mathrm{d}t.
\end{equation}
Clearly, control of either \eqref{katoVersion} or \eqref{baeVersion} achieves the desired goals: the data is not required
to be in the space $\mathscr{Y},$ but the solution at positive times does possess this higher regularity.  We find the
version \eqref{baeVersion} to be more robust in that it gives better gain of regularity.  For example using a version
of \eqref{katoVersion}, Cannone and Karch found a gain of less than one derivative at positive times
for the three-dimensional Navier-Stokes equations starting from pseudomeasure data \cite{cannoneKarch},
while in a recent work the authors have demonstrated that a version of \eqref{baeVersion} gives the full gain of two derivatives
\cite{ALN6}.

The plan of the paper is as follows.  We introduce the function spaces we will use, which are weighted Wiener spaces
and pseudomeasure spaces, in Section \ref{functionSpacesSection}.  We give abstract results for both our first method and
our second method in Section \ref{abstractResultsSection}.  We apply the first method to prove existence and analyticity
of solutions of the generalized one-dimensional Kuramoto-Sivashinsky equation in Section \ref{ottoSection}.
We apply the second method to the dissipative Constantin-Lax-Majda equation in Section \ref{CLMSection}.
We then apply the second method to the dissipative generalized Constantin-Lax-Majda equation (i.e., generalized to
include advection) in
Section \ref{gCLMSection}.  We conclude with some final remarks in Section \ref{discussionSection}, including remarks
on critical spaces for the dissipative Constantin-Lax-Majda equation.

\section{Function spaces}\label{functionSpacesSection}

We will use the notation $\mathbb{Z}_{*}$ to denote the nonzero integers, i.e. $\mathbb{Z}_{*}=\mathbb{Z}\setminus\{0\}.$

We define the space $B_{0}$ to be the Wiener algebra, which is the space of functions with Fourier series in $\ell^{1}.$
For $f\in B_{0},$ the norm is
\begin{equation*}
\|f\|_{B_{0}}=\sum_{k\in\mathbb{Z}}|\hat{f}(k)|.
\end{equation*}
We define a space-time version of the Wiener algebra with parameter $\alpha>0,$ which we call $\mathcal{B}_{\alpha}.$
For $f\in\mathcal{B}_{\alpha},$ the norm is
\begin{equation*}
\|f\|_{\mathcal{B}_{\alpha}}=\sum_{k\in\mathbb{Z}}\sup_{t\in[0,\infty)}e^{\alpha t|k|}|\hat{f}(k,t)|.
\end{equation*}
The space $B_{0}$ is famously an algebra, and this algebra property is inherited by $\mathcal{B}_{\alpha}.$  Specifically,
for any $f\in\mathcal{B}_{\alpha},$ for any $g\in\mathcal{B}_{\alpha},$ the product $fg$ is also in $\mathcal{B}_{\alpha},$ with
estimate
\begin{equation}\label{alphaAlgebra}
\|fg\|_{\mathcal{B}_{\alpha}}\leq c\|f\|_{\mathcal{B}_{\alpha}}\|g\|_{\mathcal{B}_{\alpha}}.
\end{equation}
We also will use a version of the $\mathcal{B}_{\alpha}$ spaces adapted to finite time intervals.  Given $T>0,$
we introduce $\mathcal{B}_{\alpha,T},$ with the norm
\begin{equation*}
\|f\|_{\mathcal{B}_{\alpha,T}}=\sum_{k\in\mathbb{Z}}\sup_{t\in[0,T]}e^{\alpha t|k|}|\hat{f}(k,t)|.
\end{equation*}
This is again an algebra:
\begin{equation*}
\|fg\|_{\mathcal{B}_{\alpha,T}}\leq c\|f\|_{\mathcal{B}_{\alpha,T}}\|g\|_{\mathcal{B}_{\alpha,T}}.
\end{equation*}

We will define four other scales of function spaces on space-time (with semi-infinite time interval),
which we will call $\mathcal{X}^{s},$
$\mathcal{Y}^{s},$ $\mathcal{Z}^{s},$ and $\mathcal{PM}^{s}.$
These differ in how the norm is calculated with respect to time and with respect to space.
First, we introduce the two of these which are based on the Wiener algebra.
Unlike the $\mathcal{B}_{\alpha}$ spaces, these do not incorporate an exponential weight.
First, the norm for the space $\mathcal{Y}^{s}$ is
\begin{equation*}
\|f\|_{\mathcal{Y}^{s}}=\sum_{k\in\mathbb{Z}}\sup_{t\geq0}|k|^{s}|\hat{f}(k,t)|.
\end{equation*}
By contrast, the spaces $\mathcal{X}^{s}$ take the integral with respect to time:
\begin{equation*}
\|f\|_{\mathcal{X}^{s}}=\sum_{k\in\mathbb{Z}}\int_{0}^{\infty}|k|^{s}|\hat{f}(k,t)|\ \mathrm{d}t.
\end{equation*}
We will also need the time-independent version of the spaces $\mathcal{Y}^{s},$ which we call
$Y^{s}.$  The norm for these spaces is
\begin{equation*}
\|f\|_{Y^{s}}=\sum_{k\in\mathbb{Z}}|k|^{s}|\hat{f}(k)|.
\end{equation*}
Of course, if $s=0,$ then we have $Y^{0}=B_{0},$ which is the Wiener algebra.

The spaces $\mathcal{PM}^{s}$ and $\mathcal{Z}^{s}$ are similar, but take the weighted Fourier series in $\ell^{\infty}$
rather than in $\ell^{1}.$  The norm for the pseudomeasure space $\mathcal{PM}^{s}$ is
\begin{equation*}
\|f\|_{\mathcal{PM}^{s}}=\sup_{k\in\mathbb{Z}}\sup_{t\geq0}|k|^{s}|\hat{f}(k,t)|.
\end{equation*}
The time-independent version of this is $PM^{s},$ with norm given by
\begin{equation*}
\|f\|_{PM^{s}}=\sup_{k\in\mathbb{Z}}|k|^{s}|\hat{f}(k)|.
\end{equation*}
The space $\mathcal{Z}^{s}$ is similar but takes an integral in time rather than a supremum in time; its norm is
\begin{equation*}
\|f\|_{\mathcal{Z}^{s}}=\sup_{k\in\mathbb{Z}}\int_{0}^{\infty}|k|^{s}|\hat{f}(k,t)|\ \mathrm{d}t.
\end{equation*}

Note that for $s\neq0$ technically these are not norms in that constant functions would have zero norm.
However we will be applying
these exclusively to periodic functions with zero mean, and these are norms on the appropriate subspaces, then.
Furthermore, when considering zero mean functions, there is no problem taking $s<0.$

\section{Abstract results}\label{abstractResultsSection}

In this section, we state two fixed point results

\subsection{First fixed point result}\label{firstFixedPointSection}
The following fixed point result is inspired by the work of Duchon and Robert
for vortex sheets in incompressible fluids \cite{duchonRobert}.  They do not state an abstract result, but the following
framework is implicit in their results.  This method has been applied to other interfacial flow problems
\cite{beckSosoeWong}, \cite{milgromAmbrose},
a problem in epitaxial growth of thin films \cite{BLMS}, mean field games \cite{CRAS2}, \cite{JMPA},
ill-posed Boussinesq equation of Bona-Chen-Saut type \cite{ambroseBonaMilgrom}, and the Kuramoto-Sivashinsky equation
\cite{ambroseMazzucato}.

Let $X,$ $\mathcal{X},$ and $\mathcal{Y}$ be Banach spaces.  Let $S:X\rightarrow\mathcal{X},$
and let $I^{+}:\mathcal{Y}\rightarrow\mathcal{X}$ be a bounded linear operator.
Also, let $N:\mathcal{X}\rightarrow\mathcal{Y}.$  Here, $S$ is meant to represent the semigroup associated
to a linear evolution, and $N$ is meant to represent the nonlinearity in a nonlinear evolution equation.
We state a Lipschitz condition on $N.$

\begin{assumption}\label{firstAssumption}
There exists $f:\mathbb{R}^{+}\times\mathbb{R}^{+}\rightarrow\mathbb{R}^{+}$ such that
$f$ is continuous and monotone in each argument,
and $f(0,0)=0,$ and such that for all $u\in\mathcal{X},$ for all $v\in\mathcal{X},$
$\|N(u)-N(v)\|_{\mathcal{Y}}\leq f(\|u\|_{\mathcal{X}},\|v\|_{\mathcal{X}})\|u-v\|_\mathcal{X}.$
\end{assumption}

Under this assumption, we prove the following fixed point result.

\begin{lemma}\label{firstAbstractTheorem}
Assume Assumption \ref{firstAssumption}.
If there exists $r_{1}>\|I^{+}\|_{op}\|N(0)\|_{\mathcal{Y}}$ and there exists $r_{0}>0$ such that
$f(r_{0}+r_{1},0)(r_{0}+r_{1})\leq \frac{r_{1}}{\|I^{+}\|_{op}}-\|N(0)\|_{\mathcal{Y}},$
and if $\|I^{+}\|_{op}f(r_{0}+r_{1},r_{0}+r_{1})<1,$ then for any $u_{0}\in X$ such that
$\|Su_{0}\|_\mathcal{X}\leq r_{0},$ there exists a solution of
\begin{equation*}
u=Su_{0}+I^{+}(N(u)).
\end{equation*}
\end{lemma}

\begin{proof}  Fix $u_{0}\in X$ such that $\|Su_{0}\|_{\mathcal{X}}\leq r_{0}.$
We let $\mathcal{X}_{B}$ be the closed ball in $\mathcal{X}$ centered at $Su_{0},$ with radius $r_{1}.$
We let the mapping $\mathcal{T}$ be given by $\mathcal{T}(u)=Su_{0}+I^{+}(N(u)).$
We will show that $\mathcal{T}$ is a contraction on $\mathcal{X}_{B}.$  Let $u\in\mathcal{X}_{B}.$
Clearly $\mathcal{T}(u)\in\mathcal{X};$ we first demonstrate that $\mathcal{T}(u)\in\mathcal{X}_{B}.$

We use the fact that $I^{+}$ is a bounded linear operator,
\begin{equation}\label{startingXToX}
\|\mathcal{T}(u)-Su_{0}\|_{\mathcal{X}}=\|I^{+}N(u)\|_{\mathcal{X}}
\leq \|I^{+}\|_{op}\|N(u)\|_{\mathcal{Y}}.
\end{equation}
Next we use the Lipschitz assumption on $N$ and the triangle inequality to say
\begin{equation*}
\|N(u)\|_{\mathcal{Y}}\leq f(\|u\|_{\mathcal{X}},0)\|u\|_{\mathcal{X}}+\|N(0)\|_{\mathcal{Y}}.
\end{equation*}
Using this with \eqref{startingXToX}, we have
\begin{equation*}
\|\mathcal{T}(u)-Su_{0}\|_{\mathcal{X}}\leq \|I^{+}\|_{op}(f(\|u\|_{\mathcal{X}},0)\|u\|_{\mathcal{X}}+\|N(0)\|_{\mathcal{Y}}).
\end{equation*}
For any $u\in\mathcal{X}_{B},$ we have $\|u\|_{\mathcal{X}}\leq r_{0}+r_{1}.$
Since $f$ is monotone in its first argument by assumption, we then find
\begin{equation*}
\|\mathcal{T}(u)-Su_{0}\|_{\mathcal{X}}\leq \|I^{+}\|_{op}(f(r_{0}+r_{1},0)(r_{0}+r_{1})+\|N(0)\|_{\mathcal{Y}}).
\end{equation*}
Using the assumption on $f(r_{0}+r_{1},0)(r_{0}+r_{1}),$ this reduces to
\begin{equation*}
\|\mathcal{T}(u)-Su_{0}\|_{\mathcal{X}}\leq r_{1},
\end{equation*}
which means that we have concluded that $\mathcal{T}$ maps $\mathcal{X}_{B}$ to itself.

We next demonstrate the contracting property.  For $u$ and $v$ in $\mathcal{X}_{B},$ we have
\begin{equation*}
\|\mathcal{T}(u)-\mathcal{T}(v)\|_{\mathcal{X}}=\|I^{+}(N(u)-N(v))\|_{\mathcal{X}}\leq \|I^{+}\|_{op}\|N(u)-N(v)\|_{\mathcal{Y}}.
\end{equation*}
We then use the Lipschitz assumption on $N$ and the monotonicity of $f,$ arriving at
\begin{equation*}
\|\mathcal{T}(u)-\mathcal{T}(v)\|_{\mathcal{X}}\leq\|I^{+}\|_{op}f(r_{0}+r_{1},r_{0}+r_{1})\|u-v\|_{\mathcal{X}}.
\end{equation*}
Since we have assumed that $\|I^{+}\|_{op}f(r_{0}+r_{1},r_{0}+r_{1})<1,$ this is the desired contracting property.

The contraction mapping theorem applies, and the proof is complete.
\end{proof}

The lemma clarifies somewhat when $N(0)=0,$ and that is the subject of the following corollary.
\begin{corollary}\label{DRCorollary}
Assume $N(0)=0,$ and assume Assumption \ref{firstAssumption}.
There exists $r_{0}>0$ such that for all $u_{0}\in X$ satisfying
$\|Su_{0}\|_{\mathcal{X}}\leq r_{0},$  there exists a solution $u\in\mathcal{X}$
of
\begin{equation*}
u=Su_{0}+I^{+}(N(u)).
\end{equation*}
\end{corollary}
\begin{proof} To apply Lemma \ref{firstAbstractTheorem}, we need to show that there exists $r_{0}>0$ and $r_{1}>0$ such
that two conditions hold,
\begin{equation*}
f(r_{0}+r_{1},0)(r_{0}+r_{1})\leq \frac{r_{1}}{\|I^{+}\|_{op}},\qquad
\mathrm{and}
\qquad
\|I^{+}\|_{op}f(r_{0}+r_{1},r_{0}+r_{1})<1.
\end{equation*}
We take $r_{0}=r_{1},$ and then we must find an $r_{0}$ such that the two conditions
\begin{equation*}
f(2r_{0},0)\leq \frac{1}{2\|I^{+}\|_{op}},\qquad
\mathrm{and}
\qquad
\|I^{+}\|_{op}f(2r_{0},2r_{0})<1,
\end{equation*}
hold.  Since $f$ is continuous and $f(0,0)=0,$ it is possible to find a sufficiently small $r_{0}$ such that these conditions
are satisfied.  The proof is complete.
\end{proof}

\subsection{Second fixed point result}

%\subsection{Application to parabolic PDE}
%
%Consider the partial differential equation
%\begin{equation*}
%u_{t}=Au+N(u),
%\end{equation*}
%for $x\in\mathbb{T}$ and with initial condition $u(\cdot,0)=u_{0}.$
%We may represent solutions of this initial value problem by the Duhamel formula,
%\begin{equation*}
%u=e^{At}u_{0}+\int_{0}^{t}e^{(t-s)A}N(u(\cdot,s))\ \mathrm{d}s.
%\end{equation*}

The authors have used the following abstract result in the papers \cite{ALN5}, \cite{ALN4}, and \cite{ALN6}.
\begin{lemma} \label{fixedpoint}
Let ($X,$ $\vertiii{\cdot}_{X}$) be a Banach space. Assume that $\mathcal{B}:X \times X \to X$ is a continuous bilinear operator and let $\eta>0$ satisfy $\eta\geq \|\mathcal{B}\|_{X\times X\rightarrow X}$. Then, for any $x_0 \in X$ such that
\[4\eta \vertiii{x_0}_{X}<1,\]
there exists one and only one solution to the equation
\[x=x_0+\mathcal{B}(x,x) \qquad \text{ with } \vertiii{x}_{X} < \frac{1}{2\eta}.\]
Moreover, $\vertiii{x}_{X} \leq 2\vertiii{x_0}_{X}$.
\end{lemma}
We do not provide proof here, but we instead refer the reader to \cite[p. 37, Lemma 1.2.6]{Cannone1995} and \cite{AuscherTchamitchian1999}, \cite{Cannone2003}.

\section{The 1D generalized Kuramoto-Sivashinsky equation}\label{ottoSection}

%The analyticity of solutions of certain members of the family of generalized Kuramoto-Sivashinsky equations was
%studied by Papageorgiou, Smyrlis, and Tomlin in
%\cite{papageorgiouEtAl}.  In particular, the choice $a=b/2$ is made in \eqref{ottoFamily}, and analyticity of the solutions of
%the
%resulting one-parameter family of equations (called there the \emph{Otto equations})
%is studied both analytically and numerically.  The case $b=1$ is identified
%as a borderline case for analyticity.  Naturally one would not expect analyticity for $b<1;$ in this case, one cannot
%expect exponential decay of the Fourier series, although one could study the associated Gevrey regularity of solutions.
%For $b>1,$ Papageorgiou et al. find that solutions after a long time settle down to be analytic with a certain radius
%$\rho(b)>0,$ but that $\rho(b)\rightarrow0$ as $b\rightarrow 1+.$  Granero-Belinch\'{o}n and Hunter have also studied
%this problem, giving a proof of analyticity of solutions for $b>1$ \cite{graneroBelinchonHunter}.  A related problem
%has been studied by Kiselev, Nazarov, and Shterenberg, showing analyticity for $b\geq1$ but in the absence of the linearly
%destabilizing term (i.e., the $\Lambda^{a}$ term is not present at all) \cite{kiselevEtAl}.

As discussed in the introduction, we will now show that small solutions of \eqref{ottoFamily} starting from Wiener
data instantly gain analyticity.  We consider this to be rough data, as functions in the Wiener algebra are continuous
but not generally better than continuous.
We treat the parameter range $a\geq1,$ $0<b<a.$
We show that in some cases (no linearly growing modes), the solution is global and the
radius of analyticity grows without bound.  In the general case (in which linearly growing modes are present),
the result we present here is only applies for a short time,
but still over the relevant time interval solutions are analytic after the initial time.  Note that the Papageorgiou et al. \cite{papageorgiouEtAl}
results only considered analyticity over long time intervals, so only the case of no linearly growing modes is directly
relevant to their results.

We consider the equation
\begin{equation}\label{mainEquation}
u_{t}=u^{m}u_{x}-\nu\Lambda^{b} u +\Lambda^{a}u,
\end{equation}
where the operator $\Lambda$ can be represented either as $\Lambda=\sqrt{-\partial_{x}^{2}}=H\partial_{x}$ or through its Fourier symbol as
$\hat{\Lambda}(k)=|k|.$  We consider the spatially periodic case, so $x\in\mathbb{T},$ and $H$ is the periodic Hilbert transform.
The initial data is
\begin{equation}\label{data}
u(\cdot,0)=u_{0}.
\end{equation}
Since the evolution \eqref{mainEquation} conserves the mean of solutions, we consider zero-mean solutions (so we assume
that the mean of $u_{0}$ is equal to zero).
We write the Duhamel formula for \eqref{mainEquation},
\begin{equation*}
u=e^{(-\nu\Lambda^{b}+\Lambda^{a})t}u_{0}
+\frac{1}{m+1}\int_{0}^{t}e^{(-\nu\Lambda^{b}+\Lambda^{a})(t-s)}\partial_{x}((u(\cdot,s))^{m+1})\ \mathrm{d}s.
\end{equation*}

\subsection{Global theorem}  We work in two cases: first, we demonstrate global existence of solutions when $\nu>1;$
this is the case that there are no linearly growing Fourier modes.  Our result will include that the radius of analyticity of
solutions grows linearly in time, for all time.  Solutions may exist for all time even in the complementary case $\nu\leq1,$ but
the radius of analyticity of solutions would not be expected to grow without bound in that case.  For the remainder of this
subsection, we fix $\nu>1.$

To apply the abstract result of Section \ref{firstFixedPointSection},
we write this as $u=Su_{0}+I^{+}(N(u)),$ where $S,$ $N,$ and $I^{+}$
are defined as
\begin{equation*}
Su_{0}=e^{(-\nu\Lambda^{b}+\Lambda^{a})t}u_{0}, \qquad N(u)=\frac{1}{m+1}u^{m+1},
\end{equation*}
\begin{equation*}
(I^{+}f)(\cdot,t)=\int_{0}^{t}e^{(-\nu\Lambda^{b}+\Lambda^{a})(t-s)}\partial_{x}f(\cdot,s)\ \mathrm{d}s.
\end{equation*}
Note that we clearly have $N(0)=0,$ and thus we wish to apply Corollary \ref{DRCorollary}.
To do so, we must do the following:
\begin{itemize}
\item Identify the spaces $X,$ $\mathcal{X},$ and $\mathcal{Y},$
\item Show that the linear operator $S$ maps $X$ to $\mathcal{X},$
\item Show that the linear operator $I^{+}$ maps $\mathcal{Y}$ to $\mathcal{X}$ and is bounded,
\item And, show that there exists a continuous, monotone $f$ such that $f(0,0)=0$ and
\begin{equation*}
\|N(u)-N(v)\|_{\mathcal{Y}}\leq f(\|u\|_{\mathcal{X}},\|v\|_{\mathcal{X}})\|u-v\|_{\mathcal{X}}.
\end{equation*}
\end{itemize}

We have introduced the spaces $B_{0}$ and $\mathcal{B}_{\alpha}$ in Section \ref{functionSpacesSection} above.
We let $X=B_{0}.$
We let $\alpha\in(0,\nu-1)$ be given, and we take $\mathcal{X}=\mathcal{Y}=\mathcal{B}_{\alpha}.$
We make a couple of remarks on the choice of the space $\mathcal{B}_{\alpha}.$
Since functions in $\mathcal{B}_{\alpha}$ have
domain $\mathbb{T}\times[0,\infty),$ any solution we prove to exist in $\mathcal{B}_{\alpha}$ is automatically a global
solution.
Functions in $\mathcal{B}_{\alpha}$ have Fourier series which decay exponentially at all $t>0,$ with exponential decay rate
$\alpha t.$  This implies that at each time $t>0,$ the functions are spatially analytic, with radius of analyticity at least
$\alpha t.$  Thus by using $\mathcal{B}_{\alpha},$ we will demonstrate global existence and analyticity simultaneously.

%We define the operator $\mathcal{T}$ by the right-hand side of this equation, i.e.
%\begin{equation*}
%\mathcal{T}u=e^{(-\nu\Lambda+\Lambda^{1/2})t}u_{0}
%+\frac{1}{2}\int_{0}^{t}e^{(-\nu\Lambda+\Lambda^{1/2})(t-s)}\partial_{x}((u(\cdot,s))^{2})\ \mathrm{d}s.
%\end{equation*}
%Then, solutions of the initial value problem \eqref{mainEquation}, \eqref{data} are fixed points of the operator $\mathcal{T}.$
%We will demonstrate existence of a solution by the contraction mapping theorem.  For this we need a suitable function space.

We now consider boundedness of the semigroup.
Specifically, we take $u_{0}\in B_{0}$ (with mean zero)
and we demonstrate that $e^{(-\nu\Lambda^{b}+\Lambda^{a})t}u_{0}$ is in
$\mathcal{B}_{\alpha}.$
We will rely upon our previously stated conditions $\nu>1$ and $\alpha\in(0,\nu-1),$ as well as $b\geq1$ and $0<a<b.$
We compute the norm:
\begin{multline}\nonumber
\|e^{(-\nu\Lambda^{b}+\Lambda^{a})t}u_{0}\|_{\mathcal{B}_{\alpha}}
=\sum_{k\in\mathbb{Z}_{*}}\sup_{t\in[0,\infty)}e^{\alpha t|k|}|e^{-\nu t|k|^{b}+t|k|^{a}}\hat{u}_{0}(k)|
\\
=\sum_{k\in\mathbb{Z}_{*}}\sup_{t\in[0,\infty)}
e^{\alpha t(|k|-|k|^{b})}e^{(\alpha-\nu)t|k|^{b}+t|k|^{a}}|\hat{u}_{0}(k)|.
\end{multline}
The first exponent on the right-hand side is non-positive since $b\geq1.$
The second exponent is also non-positive since $\alpha\in(0,\nu-1);$ this second exponent
may be written $t|k|^{a}((\alpha-\nu)|k|^{b-a}+1).$
Therefore the supremum is attained at $t=0$ for every $k\in\mathbb{Z}_{*}.$  We therefore have
\begin{equation*}
\|e^{(-\nu\Lambda^{b}+\Lambda^{a})t}u_{0}\|_{\mathcal{B}_{\alpha}}
=\sum_{k\in\mathbb{Z}_{*}}|\hat{u}_{0}(k)|=\|u_{0}\|_{B_{0}}.
\end{equation*}

%Let $X$ denote the ball in $\mathcal{B}_{\alpha}$ centered at $e^{(-\nu\Lambda+\Lambda^{1/2})t}u_{0},$ with radius $r_{1}.$
%Denote by $r_{0}=\|e^{(-\nu\Lambda+\Lambda^{1/2})t}u_{0}\|_{\mathcal{B}_{\alpha}}=\|u_{0}\|_{B_{0}}.$
%We will show that for appropriate choices of $r_{0}$ and $r_{1},$ the mapping $\mathcal{T}$ is a contraction on $X.$
%We must first show that $\mathcal{T}$ maps $X$ to $X.$  Note that for any $f\in X,$ we have $\|f\|_{\mathcal{B}_{\alpha}}\leq r_{0}+r_{1}.$
%
%To show that $\mathcal{T}$ maps $X$ to $X,$ we must show that, given $f\in X,$ we have
%\begin{equation*}
%\left\|\frac{1}{2}\int_{0}^{t}e^{(-\nu\Lambda+\Lambda^{1/2})(t-s)}\partial_{x}((u(\cdot,s))^{2})\ \mathrm{d}s\right\|_{\mathcal{B}_{\alpha}}\leq r_{1}.
%\end{equation*}
%To help in demonstrating this, we introduce the operator $I^{+}.$
%For any $f\in\mathcal{B}_{\alpha},$ we define
%\begin{equation*}
%(I^{+}f)(\cdot,t)=\int_{0}^{t}e^{(-\nu\Lambda+\Lambda^{1/2})(t-s)}\partial_{x}f(\cdot,s)\ \mathrm{d}s.
%\end{equation*}

We next demonstrate the boundedness of the linear operator $I^{+}.$
%We compute an upper bound for the operator norm of $I^{+}$ as a map from $\mathcal{B}_{\alpha}$ to $\mathcal{B}_{\alpha}:$
We begin by writing an expression for the norm:
\begin{equation*}
\|I^{+}f\|_{\mathcal{B}_{\alpha}}
=\sum_{k\in\mathbb{Z}}\sup_{t\in[0,t)}e^{\alpha t|k|}\left|\int_{0}^{t}e^{(-\nu |k|^{b}+|k|^{a})(t-s)}ik\hat{f}(k,s)|\ \mathrm{d}s\right|.
\end{equation*}
We note that because of the factor of $ik$ present on the right-hand side, we may exclude the $k=0$ mode from the summation.
We also move the absolute value inside the integral:
\begin{equation*}
\|I^{+}f\|_{\mathcal{B}_{\alpha}}\leq\sum_{k\in\mathbb{Z}_{*}}\sup_{t\in[0,\infty)}e^{\alpha t|k|}\int_{0}^{t}
e^{(-\nu |k|^{b}+|k|^{a})(t-s)}|k||\hat{f}(k,s)|\ \mathrm{d}s.
\end{equation*}
We rearrange factors of exponentials, and we also take another supremum inside the integral:
\begin{multline}\nonumber
\|I^{+}f\|_{\mathcal{B}_{\alpha}}
\\
\leq\sum_{k\in\mathbb{Z}_{*}}\sup_{t\in[0,\infty)}|k|
e^{(\alpha|k|-\nu|k|^{b}+|k|^{a})t}\int_{0}^{t}e^{(-\alpha|k|+\nu|k|^{b}-|k|^{a})s}
\left(e^{\alpha s|k|}|\hat{f}(k,s)|\right)\ \mathrm{d}s
\\
\leq \sum_{k\in\mathbb{Z}_{*}}\sup_{t\in[0,\infty)}|k|
e^{(\alpha|k|-\nu|k|^{b}+|k|^{a})t}\int_{0}^{t}e^{(-\alpha|k|+\nu|k|^{b}-|k|^{a})s}
\left(\sup_{\tau\in[0,\infty)}e^{\alpha\tau |k|}|\hat{f}(k,\tau)|\right)\ \mathrm{d}s.
\end{multline}
This supremum with respect to $\tau$ may now be pulled through the integral and through the supremum with respect to $t.$
This yields
\begin{multline}\nonumber
\|I^{+}f\|_{\mathcal{B}_{\alpha}}\leq\sum_{k\in\mathbb{Z}_{*}}\left(\sup_{t\in[0,\infty)}e^{\alpha t|k|}|\hat{f}(k,t)|\right)\cdot
\\
\cdot\left(\sup_{t\in[0,\infty)}|k|
e^{(\alpha|k|-\nu|k|^{b}+|k|^{a})t}\int_{0}^{t}e^{(-\alpha|k|+\nu|k|^{b}-|k|^{a})s}\ \mathrm{d}s\right).
\end{multline}
We introduce another supremum on the right-hand side, yielding
\begin{multline}\nonumber
\|I^{+}f\|_{\mathcal{B}_{\alpha}}\leq\left(\sum_{k\in\mathbb{Z}_{*}}\sup_{t\in[0,\infty)}e^{\alpha t|k|}|\hat{f}(k,t)|\right)\cdot
\\
\cdot\left(\sup_{k\in\mathbb{Z}_{*}}\sup_{t\in[0,\infty)}|k|
e^{(\alpha|k|-\nu|k|^{b}+|k|^{a})t}\int_{0}^{t}e^{(-\alpha|k|+\nu|k|^{b}-|k|^{a})s}
\ \mathrm{d}s\right)
\\
\leq \|f\|_{\mathcal{B}_{\alpha}}
\left(\sup_{k\in\mathbb{Z}_{*}}\sup_{t\in[0,\infty)}|k|
e^{(\alpha|k|-\nu|k|^{b}+|k|^{a})t}\int_{0}^{t}e^{(-\alpha|k|+\nu|k|^{b}-|k|^{a})s}
\ \mathrm{d}s\right).
\end{multline}
Next, we compute the integral on the right-hand side, finding
\begin{equation}\label{cancellation}
\|I^{+}f\|_{\mathcal{B}_{\alpha}}\leq\|f\|_{\mathcal{B}_{\alpha}}\left(\sup_{k\in\mathbb{Z}_{*}}\sup_{t\in[0,\infty)}|k|
e^{(\alpha|k|-\nu|k|^{b}+|k|^{a})t}
\cdot
\frac{e^{(-\alpha|k|+\nu|k|^{b}-|k|^{a})t}-1}{-\alpha|k|+\nu|k|^{b}-|k|^{a}}\right).
\end{equation}
This simplifies as
\begin{equation*}
\|I^{+}f\|_{\mathcal{B}_{\alpha}}\leq\|f\|_{\mathcal{B}_{\alpha}}\left(\sup_{k\in\mathbb{Z}_{*}}\sup_{t\in[0,\infty)}
\frac{1-e^{(\alpha|k|-\nu|k|^{b}+|k|^{a})t}}{\nu|k|^{b-1}-|k|^{a-1}-\alpha}\right)\leq\frac{\|f\|_{\mathcal{B}_{\alpha}}}{\nu-1-\alpha}.
\end{equation*}
Thus, the operator norm of $I^{+}$ (as a map from $\mathcal{B}_{\alpha}$ to itself) is at most $1/(\nu-1-\alpha).$

Finally, we must demonstrate the Lipschitz property for the nonlinearity, $N.$  This is straightforward, though, because
of the algebra property of the space $\mathcal{B}_{\alpha}.$  Specifically, we have
\begin{equation*}
\|N(u)-N(v)\|_{\mathcal{B}_{\alpha}}=\frac{1}{m+1}\|u^{m+1}-v^{m+1}\|_{\mathcal{B}_{\alpha}}
\leq f(\|u\|_{\mathcal{B}_{\alpha}},\|v\|_{\mathcal{B}_{\alpha}})\|u-v\|_{\mathcal{B}_{\alpha}},
\end{equation*}
where $f$ is given by
\begin{equation*}
f(w_{1},w_{2})=\frac{c^{m}}{m+1}\sum_{j=0}^{m}w_{1}^{j}w_{2}^{m-j}.
\end{equation*}
Here, we have made $m$ applications of \eqref{alphaAlgebra}, which explains the appearance of the constant $c^{m}.$
This $f$ clearly satisfies the required properties.
We conclude that Corollary \ref{DRCorollary} applies, and we have the following result.

We have proved the following theorem.
\begin{theorem}\label{refereeTheorem}
Let $b\geq1,$ $0<a<b,$ $\nu>1$ and $\alpha\in(0,\nu-1)$ be given.    There exists $\varepsilon>0$ such that for all $u_{0}\in B_{0}$ such that
$\|u_{0}\|_{B_{0}}<\varepsilon,$
there exists a unique $u\in\mathcal{B}_{\alpha}$ such
that $u$ is a global
mild solution of the initial value problem \eqref{mainEquation}, \eqref{data}.  This solution is analytic for all
$t>0,$ with radius of analyticity at least $\alpha t.$
\end{theorem}

\subsection{Local theorem}

The previous theorem showed that for $\nu>1,$ solutions to the initial value problem for \eqref{mainEquation}
exists for $t\in[0,\infty),$ as long as the data is sufficiently small.  We now
let $\nu>0$ be given, and $T>0.$  We show that a solution to the initial value problem for \eqref{mainEquation}
exists on $[0,T]$ for sufficiently small data.  The size of the data will go to zero as $\nu\rightarrow0$ or as
$T\rightarrow\infty.$  We again will use Corollary \ref{DRCorollary}, but now with the function spaces
$\mathcal{B}_{\alpha,T}$ rather than the previous $\mathcal{B}_{\alpha}.$
We only require $\alpha>0$ (unless $b=1,$ in which case we also require $\alpha<\nu$); let such an $\alpha$ be given.

We must repeat the steps that the semigroup maps $B_{0}$ to $\mathcal{B}_{\alpha,T},$ and that $I^{+}$ is a bounded
linear operator from $\mathcal{B}_{\alpha,T}$ to $\mathcal{B}_{\alpha,T}.$
We let $M_{1}>0$ be defined by
\begin{equation*}
M_{1}=\sup_{k\in\mathbb{Z}_{*}}\sup_{t\in[0,T]}e^{(\alpha|k|-\nu|k|^{b}+|k|^{a})t}.
\end{equation*}
To see that $M_{1}$ is finite, recall that $b\geq1$ and $b>a$ (and if $b=1$ we also have $\alpha<\nu$)
so there are at most finitely many $k$ for which
the exponent is positive.  Since $T$ is finite, the exponential is bounded overall.
We introduce some related notation, letting $\Omega_{f}$ be the finite set of $k$ for which the exponent is non-negative:
\begin{equation*}
\Omega_{f}=\{k\in\mathbb{Z}_{*}: \alpha|k|-\nu|k|^{b}+|k|^{a}\geq0\},
\end{equation*}
and we let the infinite set $\Omega_{i}$ be $\Omega_{i}=\mathbb{Z}^{3}_{*}\setminus\Omega_{f}.$
We let $M_{2}>0$ be such that for all $k\in\Omega_{f},$ $|k|\leq M_{2}.$
We take $M_{3}>0$ be such that for all $k\in\Omega_{i},$
\begin{equation}\label{M3Definition}
\nu|k|^{b}-\alpha|k|-|k|^{a}>M_{3}|k|^{b}.
\end{equation}

The semigroup may be estimated, then, using the constant $M_{1}:$
\begin{equation*}
\|e^{(-\nu\Lambda^{b}+\Lambda^{a})t}u_{0}\|_{\mathcal{B}_{\alpha,T}}
=\sum_{k\in\mathbb{Z}_{*}}\sup_{t\in[0,T]}e^{(\alpha|k|-\nu|k|^{b}|+|k|^{a})t}|\hat{u}_{0}(k)|
\leq M_{1}\|u_{0}\|_{B_{0}}.
\end{equation*}

We next must estimate the Duhamel integral operator, $I^{+}.$  We begin with the statement of the norm of $I^{+}h,$
for $h\in\mathcal{B}_{\alpha,T}:$
\begin{equation*}
\|I^{+}h\|_{\mathcal{B}_{\alpha,T}}=\sum_{k\in\mathbb{Z}_{*}}\sup_{t\in[0,T]}\left| e^{\alpha|k|t}\int_{0}^{t}
e^{(t-s)(-\nu|k|^{b}+|k|^{a})}ik\hat{h}(k,s)
\ \mathrm{d}s\right|.
\end{equation*}
We apply the triangle inequality to bound this as
\begin{equation*}
\|I^{+}h\|_{\mathcal{B}_{\alpha,T}}\leq\sum_{k\in\mathbb{Z}_{*}}\sup_{t\in[0,T]} e^{\alpha|k|t}\int_{0}^{t}
e^{(t-s)(-\nu|k|^{b}+|k|^{a})}|k||\hat{h}(k,s)|
\ \mathrm{d}s.
\end{equation*}
%We now decompose $\mathbb{Z}_{*}$ into two sets, one of which is finite and one of which is infinite.
%We denote $\mathbb{Z}_{*}=\Omega_{f}\cup\Omega_{i},$ with these being defined as
%\begin{equation*}
%\Omega_{f}=\{k\in\mathbb{Z}_{*}: -(\nu-\alpha)|k|+|k|^{1/2} \geq 0\},
%\qquad
%\Omega_{i}=\mathbb{Z}_{*}\setminus\Omega_{f}.
%\end{equation*}
Using this decomposition into $\Omega_{f}$ and $\Omega_{i},$ we may write the sum on the right-hand side as
\begin{multline}\nonumber
\|I^{+}h\|_{\mathcal{B}_{\alpha,T}}\leq\sum_{k\in\Omega_{f}}\sup_{t\in[0,T]} e^{\alpha|k|t}\int_{0}^{t}
e^{(t-s)(-\nu|k|^{b}+|k|^{a})}|k||\hat{h}(k,s)|\ \mathrm{d}s
\\
+
\sum_{k\in\Omega_{i}}\sup_{t\in[0,T]} e^{\alpha|k|t}\int_{0}^{t}
e^{(t-s)(-\nu|k|^{b}+|k|^{a})}|k||\hat{h}(k,s)|\ \mathrm{d}s = I + II.
\end{multline}

We then bound each of the terms $I$ and $II.$  For $I,$ first rearrange factors of exponentials to have
\begin{equation*}
I=\sum_{k\in\Omega_{f}}\sup_{t\in[0,T]}\int_{0}^{t}e^{(t-s)(\alpha|k|-\nu|k|^{b}+|k|^{a})}e^{\alpha s |k|}|\hat{h}(k,s)|\ \mathrm{d}s.
\end{equation*}
We then estimate this as
\begin{multline}\nonumber
I
\leq \left(\sum_{k\in\mathbb{Z}_{*}}\sup_{s\in[0,T]}e^{\alpha s|k|}|\hat{h}(k,s)|\right)
\cdot
\\
\cdot
\left(\sup_{k\in\Omega_{f}}\sup_{t\in[0,T]}\int_{0}^{t}|k|e^{(t-s)(\alpha|k|-\nu|k|^{b}+|k|^{a})}\ \mathrm{d}s\right)
\leq TM_{1}M_{2}\|h\|_{\mathcal{B}_{\alpha,T}}.
\end{multline}

We next estimate $II.$  We may rearrange factors of exponentials to write $II$ as
\begin{equation*}
II=\sum_{k\in\Omega_{i}}\sup_{t\in[0,T]}e^{t(\alpha|k|-\nu|k|^{b}+|k|^{a})}|k|
\int_{0}^{t}e^{s(-\alpha|k|+\nu|k|^{b}-|k|^{a})}e^{\alpha s|k|}|\hat{h}(k,s)|\ \mathrm{d}s.
\end{equation*}
We then estimate this as
\begin{multline}\nonumber
II\leq \left(\sum_{k\in\mathbb{Z}_{*}}\sup_{s\in[0,T]}e^{\alpha s|k|}|\hat{h}(k,s)|\right)
\cdot
\\
\cdot
\left(\sup_{k\in\Omega_{i}}\sup_{t\in[0,T]}e^{t(\alpha|k|-\nu|k|^{b}+|k|^{a})}|k|
\int_{0}^{t}e^{s(-\alpha|k|+\nu|k|^{b}-|k|^{a})}\ \mathrm{d}s\right).
\end{multline}
We recognize the first factor on the right-hand side as the norm of $h,$ and we compute the integral in the second
factor on the right-hand side.  These considerations yield the following:
\begin{equation*}
II\leq \|h\|_{\mathcal{B}_{\alpha,T}}\sup_{k\in\Omega_{i}}\sup_{t\in[0,T]}e^{t(\alpha|k|-\nu|k|^{b}+|k|^{a})}|k|
\left(\frac{e^{t(-\alpha|k|+\nu|k|^{b}-|k|^{a})}-1}{-\alpha|k|+\nu|k|^{b}-|k|^{a}}\right).
\end{equation*}
Since the denominator on the right-hand side is positive (for $k\in\Omega_{i}$), we may neglect the negative term in the
numerator.  Doing this and simplifying, we have
\begin{equation*}
II\leq \|h\|_{\mathcal{B}_{\alpha,T}}\sup_{k\in\Omega_{i}}\frac{|k|}{-\alpha|k|+\nu|k|^{b}-|k|^{a}}.
\end{equation*}
Finally, using \eqref{M3Definition}, this becomes
\begin{equation*}
II\leq \frac{1}{M_{3}}\|h\|_{\mathcal{B}_{\alpha,T}}.
\end{equation*}
This completes the proof that $I^{+}$ is bounded from $\mathcal{B}_{\alpha,T}$ to itself.  We now may employ
Corollary \ref{DRCorollary}, and we have proved the following theorem.

\begin{theorem}\label{localTheorem}
Let $b\geq1,$ $0<a<b,$ $\nu>0,$ $T>0,$ and $\alpha>0$ be given (with $\alpha<\nu$ if $b=1$).
There exists $\varepsilon>0$ such that if $\|u_{0}\|_{B_{0}}<\varepsilon,$
then there exists a unique $u\in\mathcal{B}_{\alpha,T}$ such that $u$ is a mild solution of
\eqref{mainEquation}, \eqref{data}.  This
solution is analytic for all $t\in(0,T]$ with radius of analyticity at least $\alpha t.$
\end{theorem}

\begin{remark} In Theorem \ref{localTheorem}, the time horizon $T$ can be taken arbitrarily large, and
the rate of gain of analyticity $\alpha$ may also be taken arbitrarily large (for $b>1$).  However these choices impact
the size of data, $\varepsilon,$ with large values of $T$ and/or $\alpha$ leading to small values of $\varepsilon.$
\end{remark}

\section{The Constantin-Lax-Majda equation with dissipation}\label{CLMSection}

In this section we consider the equation
\begin{equation}\label{CLMd}
\tilde{\omega}_{t}=\tilde{\omega} H(\tilde{\omega}) -\nu\Lambda^{\sigma}\tilde{\omega},
\end{equation}
subject to initial condition
\begin{equation}\label{CLMData}
\tilde{\omega}(\cdot,0)=\tilde{\omega}_{0},
\end{equation}
and periodic boundary conditions.  Here, $H$ is the periodic Hilbert transform.
When $\nu=0,$ this is the Constantin-Lax-Majda equation \cite{CLM}, introduced as a one-dimensional
model of vortex stretching.  It was shown in \cite{CLM} that the equation develops singularities in finite time.
With added dissipation, it has been shown that the long-time behavior depends on the spatial domain and the size
of the initial data.  When $x\in\mathbb{R},$ even small solutions may blow up in finite time
\cite{schochet}, while in the case of periodic boundary conditions sufficiently small initial data
leads to global existence \cite{ALSS}.

There are two global existence results presented in \cite{ALSS}, and they are for a more general version of
\eqref{CLMd} which has an additional advection term.  For sufficiently small data in the Wiener algebra (which in our notation
is $Y^{0}$) and
for $\sigma\geq1,$ it is proved that solutions exist for all time and are analytic at positive times.  With data instead
sufficiently small in $L^{2}$ and $\sigma>1,$ global existence is shown.  The requirements $\sigma\geq1$ are
related to the presence of the advective term which we are not considering in the current section.
We will show now that for any $\sigma>0,$ sufficiently small periodic data in $Y^{-\sigma/2}$ leads to existence
of a global solution.  This demonstrates a significant improvement over the global existence results of
\cite{ALSS}, as the two global existence results there both required decaying Fourier series, while the present
result allows the Fourier coefficients to grow as $k$ goes to infinity.  The results of \cite{ALSS} for Wiener data use the one-norm
approach, i.e. the method using Corollary \ref{DRCorollary}.  In that method, the algebra property of the spaces
$\mathcal{B}_{\alpha}$ is strongly relied upon.  The spaces $Y^{s}$ are algebras for $s\geq0$ but not for negative
values of $s.$  Thefore the extension to spaces with negative index requires a different method.  We thus use the
two-norm approach, i.e. the method related to Lemma \ref{fixedpoint}.
As we have said, the approach using Corollary \ref{DRCorollary} as in \cite{ALSS} automatically gives
analyticity of the solutions.   One would not expect
analyticity unless $\sigma\geq1,$ and in this case, for the solutions we will prove to exist using Lemma \ref{fixedpoint},
we can again show that the solutions gain analyticity at any finite time.  This must be done separately, though, and is not
part of the existence argument when using this method.

We define $\omega=\mathbb{P}_{0}\omega,$ with $\mathbb{P}_{0}$ being the projection which removes the
mean of a periodic function.  Properties of the Hilbert transform imply that for any $f,$
\begin{equation}\label{hilbertProperties}
\mathbb{P}_{0}(H(f))=H(f),\qquad \mathbb{P}_{0}(fH(f))=fH(f).
\end{equation}
Furthermore, the Fourier symbol of $H$ is $\hat{H}(k)=-i\mathrm{sgn}(k).$
Of course, for any $f$ and for any $\sigma>0$ we also have $\mathbb{P}_{0}\Lambda^{\sigma}f=\Lambda^{\sigma}f.$
We see then that $\mathbb{P}_{0}\tilde{\omega}_{t}=0,$ so that the evolution \eqref{CLMd} preserves the mean
of $\tilde{\omega}.$  We denote $\omega=\mathbb{P}_{0}\tilde{\omega},$ and
$\omega_{0}=\mathbb{P}_{0}\tilde{\omega}.$  We let the mean of $\tilde{\omega}_{0}$ be denoted by $\bar{\omega},$
so that $\tilde{\omega}_{0}=\omega_{0}+\bar{\omega}.$  Then we may write the initial value problem for $\omega$ as
\begin{equation}\label{CLMd2}
\omega_{t}=\mathbb{P}_{0}(\omega H(\omega)) + \bar{\omega}H(\omega)-\nu\Lambda^{\sigma}\omega,
\end{equation}
\begin{equation}\label{CLMData2}
\omega(\cdot,0)=\omega_{0}.
\end{equation}
Note that the presence of $\mathbb{P}_{0}$ on the nonlinear term in \eqref{CLMd2} is redundant, but it will be useful
in the bilinear estimates we will be performing.

We will now apply the fixed point result Lemma \ref{fixedpoint} to prove existence of solutions for the Constantin-Lax-Majda
equation.  To do so, we need to complete the following steps:
\begin{itemize}
\item Identify the function space $X$ from the statement of Lemma \ref{fixedpoint},
\item Identify $x_{0}$ and show $x_{0}\in X,$
\item And, establish the bilinear estimate.
\end{itemize}

To begin, the function space we will use is $\mathcal{X}^{\sigma/2}\cap\mathcal{Y}^{-\sigma/2},$ where $\sigma>0$
is the order of the diffusion operator in \eqref{CLMd}.  The norm is given by
\begin{equation*}
\vertiii{f}=\|f\|_{\mathcal{X}^{\sigma/2}}+\|f\|_{\mathcal{Y}^{-\sigma/2}}.
\end{equation*}
For $\omega_{0}\in Y^{-\sigma/2},$ the point $x_{0}$ is given by
\begin{equation*}
x_{0}=e^{-t(\nu\Lambda^{\sigma}-\bar{\omega}H)}\omega_{0}.
\end{equation*}
We next must show that $x_{0}\in\mathcal{X}^{\sigma/2}$ and $x_{0}\in\mathcal{Y}^{-\sigma/2}.$  To calculate the norm
in $\mathcal{Y}^{-\sigma/2},$ we begin with the definition of the norm,
\begin{equation*}
\|x_{0}\|_{\mathcal{Y}^{-\sigma/2}}=\sum_{k\in\mathbb{Z}_{*}}\sup_{t\geq0}|k|^{-\sigma/2}
|e^{-t(\nu|k|^{\sigma}+i\bar{\omega}\mathrm{sgn}(k))}||\hat{\omega}_{0}(k)|.
\end{equation*}
We may simplify the absolute value of the exponential, and we see that the supremum occurs at $t=0.$  This yields
\begin{equation*}
\|x_{0}\|_{\mathcal{Y}^{-\sigma/2}}=\sum_{k\in\mathbb{Z}_{*}}|k|^{-\sigma/2}|\hat{\omega}_{0}(k)|
=\|\omega_{0}\|_{Y^{-\sigma/2}}.
\end{equation*}
We next calculate the norm of $x_{0}$ in the space $\mathcal{X}^{\sigma/2},$ beginning with the definition of the norm:
\begin{equation*}
\|x_{0}\|_{\mathcal{X}^{\sigma/2}}=\sum_{k\in\mathbb{Z}_{*}}\int_{0}^{\infty}|k|^{\sigma/2}
|e^{-t(\nu|k|^{\sigma}+i\bar{\omega}\mathrm{sgn}(k))}||\hat{\omega}_{0}(k)|\ \mathrm{d}t.
\end{equation*}
We again simplify the absolute value of the exponential, and pull time-independent factors through the integral, finding
\begin{equation*}
\|x_{0}\|_{\mathcal{X}^{\sigma/2}}=\sum_{k\in\mathbb{Z}_{*}}(|k|^{\sigma/2}|\hat{\omega}_{0}(k)|)
\int_{0}^{\infty}e^{-t\nu|k|^{\sigma}}\ \mathrm{d}t.
\end{equation*}
Evaluating this integral, we find
\begin{equation*}
\|x_{0}\|_{\mathcal{X}^{\sigma/2}}=\frac{1}{\nu}\sum_{k\in\mathbb{Z}_{*}}|k|^{-\sigma/2}|\hat{\omega}_{0}(k)|
=\frac{1}{\nu}\|\omega_{0}\|_{Y^{-\sigma/2}}.
\end{equation*}

The bilinear operator associated to the evolution \eqref{CLMd2} is
\begin{equation}\label{CLMd2Bilinear}
B(u,v)=\int_{0}^{t}e^{-(t-s)(\nu\Lambda^{\sigma}-\bar{\omega}H)}\mathbb{P}_{0}(uH(v))\ \mathrm{d}s.
\end{equation}
What remains is to estimate the bilinear operator in the spaces $\mathcal{X}^{\sigma/2}$ and $\mathcal{Y}^{-\sigma/2}.$

We begin with the estimate of $B(u,v)$ in $\mathcal{Y}^{-\sigma/2}.$
We have
\begin{multline}\label{CLMd2BilinearYest}
\|B(u,v)\|_{\mathcal{Y}^{-\sigma/2}} = \\
\sum_{k\in\mathbb{Z}_{*}}\sup_{t\geq 0}
\left|
\frac{1}{|k|^{\sigma/2}}
\sum_{j\in\mathbb{Z}_{*},j\neq k}
\int_{0}^{t}
e^{-(t-s)(\nu|k|^{\sigma}+i\bar{\omega}\mathrm{sgn}(k))}
\hat{u}(k-j,s)
(-i\mathrm{sgn}(j))\hat{v}(j,s)
\ \mathrm{d}s\right|.
\end{multline}
By the triangle inequality, we may immediately bound this as
\begin{equation*}
\|B(u,v)\|_{\mathcal{Y}^{-\sigma/2}}\leq
\sum_{k\in\mathbb{Z}_{*}}\sup_{t\geq0}
\frac{1}{|k|^{\sigma/2}}
\sum_{j\in\mathbb{Z}_{*}, j\neq k}
\int_{0}^{t}e^{-(t-s)|k|^{\sigma}}|\hat{u}(k-j,s)| |\hat{v}(j,s)|\ \mathrm{d}s.
\end{equation*}
We furthermore may neglect the exponential since the exponent is negative.  We also multiply and divide by
$|k-j|^{\sigma/2}|j|^{\sigma/2},$ arriving at
\begin{multline}\label{BYToBeEstimated}
\|B(u,v)\|_{\mathcal{Y}^{-\sigma/2}}\leq
\\
\sum_{k\in\mathbb{Z}_{*}}\sup_{t\geq0}\sum_{j\in\mathbb{Z}_{*},j\neq k}\int_{0}^{t}
\left(\frac{|k-j|^{\sigma/2}}{|k|^{\sigma/2}|j|^{\sigma/2}}\right)
\frac{|\hat{u}(k-j,s)|}{|k-j|^{\sigma/2}}|j|^{\sigma/2}|\hat{v}(j,s)|\ \mathrm{d}s.
\end{multline}
The factors of $j,$ $k,$ and $k-j$ satisfy the elementary inequality
\begin{equation}\label{jkk-jInequality}
\sup_{k\neq0,j\neq0}\frac{|k-j|^{\sigma/2}}{|k|^{\sigma/2}|j|^{\sigma/2}}\leq 2^{\sigma/2}.
\end{equation}
We then estimate \eqref{BYToBeEstimated} as
\begin{multline}\nonumber
\|B(u,v)\|_{\mathcal{Y}^{-\sigma/2}}\leq
\\
\left(\sup_{j\neq 0, k\neq0}\frac{|k-j|^{\sigma/2}}{|k|^{\sigma/2}|j|^{\sigma/2}}\right)
\left(\sum_{k\in\mathbb{Z}_{*}}\sup_{s\geq0}\frac{|\hat{u}(k-j,s)|}{|k-j|^{\sigma/2}}\right)
\left(\sum_{j\in\mathbb{Z}_{*},j\neq k}
\int_{0}^{\infty} |j|^{\sigma/2}|\hat{v}(j,s)|
\ \mathrm{d}s\right)
\\
\leq 2^{\sigma/2}\|u\|_{\mathcal{Y}^{-\sigma/2}}\|v\|_{\mathcal{X}^{\sigma/2}}.
\end{multline}

We are ready to estimate the bilinear term in $\mathcal{X}^{\sigma/2}.$  We begin just by using the definition of the norm,
\begin{multline}\nonumber
\|B(u,v)\|_{\mathcal{X}^{\sigma/2}}=\sum_{k\in\mathbb{Z}_{*}}\int_{0}^{\infty}\Bigg||k|^{\sigma/2}
\int_{0}^{t}\Bigg[\sum_{j\in\mathbb{Z}_{*},j\neq k}
e^{-(t-s)(\nu|k|^{\sigma}-i\bar{\omega}\mathrm{sgn}(k))}
\cdot
\\
\cdot
\hat{u}(k-j,s)(-i\mathrm{sgn}(j))\hat{v}(j,s)\Bigg]\ \mathrm{d}s\Bigg|\ \mathrm{d}t.
\end{multline}
Applying the triangle inequality, we get
\begin{equation*}
\|B(u,v)\|_{\mathcal{X}^{\sigma/2}}\leq\sum_{k\in\mathbb{Z}_{*}}\int_{0}^{\infty}|k|^{\sigma/2}
\int_{0}^{t}\sum_{j\in\mathbb{Z}_{*},j\neq k}
e^{-(t-s)\nu|k|^{\sigma}}
|\hat{u}(k-j,s)||\hat{v}(j,s)|\ \mathrm{d}s\mathrm{d}t.
\end{equation*}
We change the order of integration, finding
\begin{equation*}
\|B(u,v)\|_{\mathcal{X}^{\sigma/2}}\leq\sum_{k\in\mathbb{Z}_{*}}\int_{0}^{\infty}|k|^{\sigma/2}
\sum_{j\in\mathbb{Z}_{*},j\neq k}|\hat{u}(k-j,s)||\hat{v}(j,s)|
\int_{s}^{\infty}e^{-(t-s)\nu|k|^{\sigma}}
\ \mathrm{d}t\mathrm{d}s.
\end{equation*}
We evaluate the last integral on the right-hand side, finding
\begin{equation*}
\|B(u,v)\|_{\mathcal{X}^{\sigma/2}}\leq\frac{1}{\nu}\sum_{k\in\mathbb{Z}_{*}}\int_{0}^{\infty}|k|^{-\sigma/2}
\sum_{j\in\mathbb{Z}_{*},j\neq k}|\hat{u}(k-j,s)||\hat{v}(j,s)|
\ \mathrm{d}s.
\end{equation*}
We adjust the factors of $|k-j|$ and $|j|:$
\begin{equation*}
\|B(u,v)\|_{\mathcal{X}^{\sigma/2}}\leq\frac{1}{\nu}\sum_{k\in\mathbb{Z}_{*}}\int_{0}^{\infty}
\sum_{j\in\mathbb{Z}_{*},j\neq k}
\left(\frac{|k-j|^{\sigma/2}}{|k|^{\sigma/2}|j|^{\sigma/2}}\right)
\frac{|\hat{u}(k-j,s)|}{|k-j|^{\sigma/2}}|j|^{\sigma/2}|\hat{v}(j,s)|
\ \mathrm{d}s.
\end{equation*}
As before, we use the inequality \eqref{jkk-jInequality}, and
then we immediately find the estimate
\begin{equation*}
\|B(u,v)\|_{\mathcal{X}^{\sigma/2}}\leq\frac{2^{\sigma/2}}{\nu}\|u\|_{\mathcal{Y}^{-\sigma/2}}\|v\|_{\mathcal{X}^{\sigma/2}}.
\end{equation*}

We have completed the application of Lemma \ref{fixedpoint}, and we have therefore proven the following theorem:
\begin{theorem}\label{CLMXYTheorem}
Let $\sigma>0$ be given.  There exists $\varepsilon>0$ such that for all $\bar{\omega}\in\mathbb{R}$ and
all $\omega_{0}\in Y^{-\sigma/2}$ with $\|\omega_{0}\|_{Y^{-\sigma/2}}<\varepsilon,$
there exists $\omega\in\mathcal{X}^{\sigma/2}\cap\mathcal{Y}^{-\sigma/2}$ such that
$\tilde{\omega}=\bar{\omega}+\omega$ is a mild solution
of the initial value problem \eqref{CLMd}, \eqref{CLMData}, with initial data $\tilde{\omega}_{0}=\bar{\omega}+\omega_{0}.$
\end{theorem}
We remark that it is implicit in the assumption that $\omega_{0}\in Y^{-\sigma/2}$ that $\omega_{0}$ has mean zero.
Also note that the solution is automatically global in time, following from the definition of the spaces $\mathcal{X}^{\sigma/2}$
and $\mathcal{Y}^{-\sigma/2}.$

We have proved existence of global solutions for the dissipative Constantin-Lax-Majda problem when the diffusion parameter
satisfies only $\sigma>0.$  However if we further restrict to $\sigma\geq1,$ then
we can demonstrate analyticity of the solutions (possibly restricting to slightly smaller data).
We delay this argument until Section \ref{analyticitySection} below, as there we will simultaneously treat the analyticity of solutions
the dissipative Constantin-Lax-Majda both with and without added advection.

\section{The Constantin-Lax-Majda equation with advection and dissipation}\label{gCLMSection}

We now consider a generalization of the Constantin-Lax-Majda equation with nonlocal advection,
\begin{equation}\label{gCLMd}
\tilde{\omega}_{t}=\tilde{\omega}H(\tilde{\omega})+a\tilde{u}\tilde{\omega}_{x}-\nu\Lambda^{\sigma}\tilde{\omega},
\end{equation}
where the advective velocity $\tilde{u}$ is given by
\begin{equation*}
\tilde{u}_{x}=H(\tilde{\omega}).
\end{equation*}
This is the same as saying $\tilde{u}=\Lambda^{-1}\mathbb{P}_{0}\tilde{\omega}.$  Of course we also take the initial data
\eqref{CLMData}, as before.

As we remarked in the introduction, this is the form of advection as introduced by Okamoto, Sakajo, and Wunsch
\cite{okamotoEtAl}.
 The question of singularity formation versus global existence in the generalized
Constantin-Lax-Majda equation with various values of the parameters $\sigma$ and $a$
has been studied previously in a number of works, especially  \cite{ALSS}, \cite{chen}, \cite{CCF}, \cite{dong},  \cite{kimEtAl}, \cite{kiselev}, \cite{liRodrigo}, \cite{ALSS2},
\cite{silvestreVicol}.

As in the case without advection, the mean of $\tilde{\omega}$ is preserved.  This uses both the previously-presented
facts \eqref{hilbertProperties}, as well as
\begin{equation*}
\mathbb{P}_{0}((\Lambda^{-1}f)f_{x})=0.
\end{equation*}
This can be seen by writing out the integral, integrating by parts, and then using \eqref{hilbertProperties}.
We again let $\omega=\mathbb{P}_{0}\tilde{\omega},$ finding the evolution equation
\begin{equation}\label{gCLMd2}
\omega_{t}=\mathbb{P}_{0}(\omega H(\omega)) + a\mathbb{P}_{0}((\Lambda^{-1}\mathbb{P}_{0}\omega)\omega_{x})
+\bar{\omega}H(\omega)-\nu\Lambda^{\sigma}\omega.
\end{equation}
We again take this with data \eqref{CLMData2}.
The bilinear term associated to the evolution \eqref{gCLMd2} is
\begin{equation}\label{e:calBil}
\mathcal{B}(u,v)=B(u,v)+aB_{adv}(u,v),
\end{equation}
with $B(u,v)$ given in \eqref{CLMd2Bilinear} and $B_{adv}(u,v)$ given by
\begin{equation}\label{e:advBil}
B_{adv}(u,v)=\int_{0}^{t}e^{-(t-s)\Lambda^{\sigma}-\bar{\omega}H)}\mathbb{P}_{0}((\Lambda^{-1}\mathbb{P}_{0}u)v_{x})
\ \mathrm{d}s.
\end{equation}

\subsection{Existence of solutions with $Y^{-\sigma/2}$ data} As in the case without advection, we are able to
prove existence of solutions in $\mathcal{Y}^{-\sigma/2}\cap\mathcal{X}^{\sigma/2},$ with data in $Y^{-\sigma/2}.$
In the case without advection we only needed to assume $\sigma>0,$ but now we will require $\sigma\geq1.$

We will estimate the new contribution to the bilinear form, $B_{adv},$ in both $\mathcal{Y}^{-\sigma/2}$ and
$\mathcal{X}^{\sigma/2}.$  We begin with stating the norm in $\mathcal{Y}^{-\sigma/2},$ and making some elementary
manipulations:
\begin{multline}\label{advectionYEstimate}
\|B_{adv}(u,v)\|_{\mathcal{Y}^{-\sigma/2}}
\\
=\sum_{k\in\mathbb{Z}_{*}}\sup_{t\geq0}|k|^{-\sigma/2}
\left|\sum_{j\in\mathbb{Z}_{*},j\neq k}\int_{0}^{t}
e^{-(t-s)(\nu|k|^{\sigma}+i\bar{\omega}\mathrm{sgn}(k))}\frac{\hat{u}(k-j,s)}{|k-j|}ij\hat{v}(j,s)\ \mathrm{d}s\right|
\\
\leq\sum_{k\in\mathbb{Z}_{*}}\sum_{j\in\mathbb{Z}_{*},j\neq k}\int_{0}^{\infty}
|k|^{-\sigma/2}\frac{|\hat{u}(k-j,s)|}{|k-j|}|j||\hat{v}(j,s)|\ \mathrm{d}s.
\end{multline}
We adjust factors of $j,$ $k,$ and $k-j,$ arriving at
\begin{equation*}
\|B_{adv}(u,v)\|_{\mathcal{Y}^{-\sigma/2}}\leq
\left(\sup_{j,k} \frac{1}{|k|}\left(\frac{|k||j|}{|k-j|}\right)^{1-\sigma/2}\right)
\|u\|_{\mathcal{Y}^{-\sigma/2}}\|v\|_{\mathcal{X}^{\sigma/2}}.
\end{equation*}
If $\sigma\leq 2,$ then we may write this supremum as
\begin{equation*}
\sup_{j,k}\frac{1}{|k|}\left(\frac{|k||j|}{|k-j|}\right)^{1-\sigma/2}
=\sup_{j,k}|k|^{1-\sigma}\left(\frac{|j|}{|k||k-j|}\right)^{1-\sigma/2}\leq 2^{1-\sigma/2}.
\end{equation*}
Note that we have used here the requirement $\sigma\geq1.$
If instead $\sigma>2,$ then we write the supremum as
\begin{equation*}
\sup_{j,k}\frac{1}{|k|}\left(\frac{|k||j|}{|k-j|}\right)^{1-\sigma/2}
=\sup_{j,k}\frac{1}{|k|}\left(\frac{|k-j|}{|k||j|}\right)^{\sigma/2-1}\leq 2^{\sigma/2-1}.
\end{equation*}
In either case, we have established the bound
\begin{equation*}
\|B_{adv}(u,v)\|_{\mathcal{Y}^{-\sigma/2}}\leq c \|u\|_{\mathcal{Y}^{-\sigma/2}}\|v\|_{\mathcal{X}^{\sigma/2}}.
\end{equation*}

Next, we estimate $B_{adv}(u,v)$ in the space $\mathcal{X}^{\sigma/2}.$  We begin with the expression of
the norm of $B_{adv}(u,v)$ in this space, and make some elementary manipulations:
\begin{multline}\nonumber
\|B_{adv}(u,v)\|_{\mathcal{X}^{\sigma/2}}
\\
=\sum_{k\in\mathbb{Z}_{*}}\int_{0}^{\infty}|k|^{\sigma/2}\left|\sum_{j\in\mathbb{Z}_{*},j\neq k}\int_{0}^{t}
e^{-(t-s)(\nu|k|^{\sigma}+i\bar{\omega}\mathrm{sgn}(k))}\frac{\hat{u}(k-j,s)}{|k-j|}ij\hat{v}(j,s)\ \mathrm{d}s\right|\ \mathrm{d}t
\\
\leq
\sum_{k\in\mathbb{Z}_{*}}\sum_{j\in\mathbb{Z}_{*},j\neq k}|k|^{\sigma/2}\int_{0}^{\infty}\int_{0}^{t}e^{-(t-s)\nu|k|^{\sigma}}
\frac{|\hat{u}(k-j,s)|}{|k-j|}|j||\hat{v}(j,s)|\ \mathrm{d}s\mathrm{d}t.
\end{multline}
We change the order of integration and evaluate the integral with respect to $t,$ finding
\begin{equation*}
\|B_{adv}(u,v)\|_{\mathcal{X}^{\sigma/2}}\leq
\frac{1}{\nu}\sum_{k\in\mathbb{Z}_{*}}\sum_{j\in\mathbb{Z}_{*},j\neq k}\int_{0}^{\infty}|k|^{-\sigma/2}
\frac{|\hat{u}(k-j,s)|}{|k-j|}|j||\hat{v}(j,s)|\ \mathrm{d}s.
\end{equation*}
The quantity on the right-hand side is the same quantity as on the right-hand side of \eqref{advectionYEstimate},
and it may be estimated in the same way.  This establishes
\begin{equation*}
\|B_{adv}(u,v)\|_{\mathcal{X}^{\sigma/2}}\leq c \|u\|_{\mathcal{Y}^{-\sigma/2}}\|v\|_{\mathcal{X}^{\sigma/2}}.
\end{equation*}

Keeping in mind the proof of Theorem \ref{CLMXYTheorem} for the Constantin-Lax-Majda equation without advection,
we have now proved the following.
\begin{theorem}\label{CLMAdvectionXYTheorem}
Let $\sigma\geq1$ and $a\in\mathbb{R}$ be given.  There exists $\varepsilon>0$ such that for all
$\bar{\omega}\in\mathbb{R}$ and for
all $\omega_{0}\in Y^{-\sigma/2}$ satisfying $\|\omega_{0}\|_{Y^{-\sigma/2}}<\varepsilon,$
there exists $\omega\in\mathcal{X}^{\sigma/2}\cap\mathcal{Y}^{-\sigma/2}$ such that
$\tilde{\omega}=\bar{\omega}+\omega$ is a mild solution
of the initial value problem \eqref{gCLMd}, \eqref{CLMData},
with initial data $\tilde{\omega}_{0}=\bar{\omega}+\omega_{0}.$
\end{theorem}

\begin{remark} The choice of spaces $\mathcal{Y}^{-\sigma/2}$ and $\mathcal{X}^{\sigma/2}$ is consistent
with the Lei-Lin result on the Navier-Stokes equations \cite{leiLin}, in the sense that there the appropriate value of $\sigma$
is $\sigma=2,$ and the Bae approach to the Lei-Lin solution gives existence in the three-dimensional versions of
$\mathcal{Y}^{-1}\cap\mathcal{X}^{1}$ \cite{ALN4},  \cite{baePAMS}.
\end{remark}

\subsection{Existence of solutions with $PM^{r}$ data for $r\leq0$}

We will prove existence of solutions in $\mathcal{PM}^{r}\cap\mathcal{Z}^{r+\sigma},$ with data in $PM^{r},$
when $\sigma>1$ and $r\in\big(\frac{1-\sigma}{2},0\big].$  This result holds for all values of $a\in\mathbb{R},$
and in particular it holds for $a=0$ and $a\neq0.$  Thus it is a result for the dissipative Constantin-Lax-Majda
equation both with and without advection.

\begin{theorem} \label{ExistsAdvPMr} Let $a\in\mathbb{R},$
$\sigma>1$ and $r\in\left(\frac{1-\sigma}{2},0\right]$ be given.
There exists $\varepsilon>0$ such that for any $\bar{\omega}\in\mathbb{R}$ and $\omega_{0}\in PM^{r}$
satisfying $\|\omega_{0}\|_{PM^{r}}<\varepsilon,$
the equation \eqref{gCLMd}
has a global mild solution $\tilde{\omega}=\bar{\omega}+\omega,$
with initial data $\tilde{\omega}_{0}=\bar{\omega}+\omega_{0},$ and with
$\omega\in\mathcal{PM}^{r}\cap\mathcal{Z}^{r+\sigma}.$
\end{theorem}

\begin{proof} We begin by showing that the semigroup
maps $PM^{r}$ into both $\mathcal{PM}^{r}$ and $\mathcal{Z}^{r+\sigma}.$  We may immediately calculate
\begin{equation}\label{e:lintermgCLMd2PMr}
\|e^{(-\nu\Lambda^{\sigma}+\bar{\omega}H)t}\omega_{0}\|_{\mathcal{PM}^{r}}
= \sup_{k\in\mathbb{Z}_{*}}\sup_{t\geq0} e^{-\nu|k|^{\sigma}t}|k|^{r}|\hat{\omega}_{0}(k)|\leq\|\omega_{0}\|_{PM^{r}}.
\end{equation}
For the corresponding estimate in $\mathcal{Z}^{r+\sigma},$ we express the norm and carry out the integral, finding
\begin{equation}\label{e:lintermgCLMd2Zrandsigma}
\|e^{(-\nu\Lambda^{\sigma}+\bar{\omega}H)t}\omega_{0}\|_{\mathcal{Z}^{r+\sigma}}
=\sup_{k\in\mathbb{Z}_{*}}\int_{0}^{\infty}|k|^{r+\sigma}e^{-\nu|k|^{\sigma}t}|\hat{\omega}_{0}(k)|\ \mathrm{d}t
=\frac{1}{\nu}\|\omega_{0}\|_{PM^{r}}.
\end{equation}

Next, we establish the bilinear estimates.
We estimate $B(u,v)$ in the space $\mathcal{PM}^{r},$ beginning with the expression of the norm of $B(u,v)$ in this space
and making some first estimates:
\begin{multline}\label{CLMTermPMr}
\|B(u,v)\|_{\mathcal{PM}^{r}}
\\
=\sup_{k\in\mathbb{Z}_{*}}\sup_{t\geq0}|k|^{r}\left|\sum_{j\in\mathbb{Z}_{*},j\neq k}\int_{0}^{t}
e^{-(t-s)(\nu|k|^{\sigma}+i\bar{\omega}\mathrm{sgn}(k))}(i\mathrm{sgn}(k-j))\hat{u}(k-j,s)\hat{v}(j,s)\ \mathrm{d}s\right|
\\
\leq \sup_{k\in\mathbb{Z}_{*}}\sum_{j\in\mathbb{Z}_{*},j\neq k}\int_{0}^{\infty}|k|^{r}|\hat{u}(k-j,s)||\hat{v}(j,s)|\ \mathrm{d}s.
\end{multline}
We arrange factors of $j,$ $k,$ and $k-j$ to find the bound
\begin{equation*}
\|B(u,v)\|_{\mathcal{PM}^{r}}\leq
\left(\sup_{k\in\mathbb{Z}_{*}}\sum_{j\in\mathbb{Z}_{*},j\neq k} \frac{|k|^{r}}{|k-j|^{r}|j|^{r+\sigma}}\right)
\|u\|_{\mathcal{PM}^{r}}\|v\|_{\mathcal{Z}^{r+\sigma}}.
\end{equation*}
As long as the quantity which multiplies the norms on the right-hand side is finite, this is the desired bound.
As shown in the appendix of \cite{ALSS}, there exists $c>0$ such that $|k-j|^{-r}\leq c(|k|^{-r}+|j|^{-r}).$  Then, we
estimate the quantity as
\begin{equation*}
\sup_{k\in\mathbb{Z}_{*}}\sum_{j\in\mathbb{Z}_{*},j\neq k}\frac{|k|^{r}}{|k-j|^{r}|j|^{r+\sigma}}
\leq c\sup_{k\in\mathbb{Z}_{*}}
\sum_{j\in\mathbb{Z}_{*},j\neq k}\left(\frac{1}{|j|^{r+\sigma}}+\frac{|k|^{r}}{|j|^{2r+\sigma}}\right).
\end{equation*}
Our choice of $r$ requires both $r\leq0$ and $2r+\sigma>1,$ so this quantity is indeed finite.
We have therefore established the needed bound for $\|B(u,v)\|_{\mathcal{PM}^{r}}.$

We next express the norm of $B(u,v)$ in the space $\mathcal{Z}^{r+\sigma},$ and apply elementary inequalities:
\begin{multline}\nonumber
\|B(u,v)\|_{\mathcal{Z}^{r+\sigma}} =
\\
\sup_{k\in\mathbb{Z}_{*}}\int_{0}^{\infty}\left| |k|^{r+\sigma} \sum_{j\in\mathbb{Z}_{*},j\neq k}\int_{0}^{t}
e^{-(t-s)(\nu|k|^{\sigma}+i\bar{\omega}\mathrm{sgn}(k))}\hat{u}(k-j,s)(-i\mathrm{sgn}(j))\hat{v}(j,s)\ \mathrm{d}s\right|
\\
\leq \sup_{k\in\mathbb{Z}_{*}}\sum_{j\in\mathbb{Z}_{*},j\neq k}\int_{0}^{\infty}\int_{0}^{t} |k|^{r+\sigma}
e^{-(t-s)\nu|k|^{\sigma}}|\hat{u}(k-j,s)| |\hat{v}(j,s)|\ \mathrm{d}s.
\end{multline}
We again change the order of integration, and evaluate the integral of the exponential, finding
\begin{equation}\label{CLMtermZrandsigmaEST}
\|B(u,v)\|_{\mathcal{Z}^{r+\sigma}} \leq \frac{1}{\nu}
\sup_{k\in\mathbb{Z}_{*}}\sum_{j\in\mathbb{Z}_{*},j\neq k}\int_{0}^{\infty} |k|^{r} |\hat{u}(k-j,s)| |\hat{v}(k-j,s)|\ \mathrm{d}s.
\end{equation}
This right-hand side is the same as the right-hand side of \eqref{CLMTermPMr}, and therefore we may then estimate it the same
way.  This concludes the bound for $B(u,v).$

We next must bound $B_{adv}(u,v).$  We begin with its norm in the space $\mathcal{PM}^{r}:$
\begin{multline}\label{toBeSeenAgain}
\|B_{adv}(u,v)\|_{\mathcal{PM}^{r}} =
\\
\sup_{k\in\mathbb{Z}_{*}}\sup_{t\geq0}\left| |k|^{r} \sum_{j\in\mathbb{Z}_{*},j\neq k}\int_{0}^{t}
e^{-(t-s)(\nu|k|^{\sigma}+i\bar{\omega}\mathrm{sgn}(k))}\frac{\hat{u}(k-j,s)}{|k-j|}ij\hat{v}(j,s)\ \mathrm{d}s\right|
\\
\leq
\sup_{k\in\mathbb{Z}_{*}}\sum_{j\in\mathbb{Z}_{*}, j\neq k}\int_{0}^{\infty} |k|^{r} \frac{|\hat{u}(k-j,s)|}{|k-j|}
|j| |\hat{v}(j,s)|\ \mathrm{d}s.
\end{multline}
We then adjust factors of $j,$ $k,$ and $k-j,$ to find
\begin{multline}\nonumber
\|B_{adv}(u,v)\|_{\mathcal{PM}^{r}}\leq
\\
\sup_{k\in\mathbb{Z}_{*}}\sum_{j\in\mathbb{Z}_{*},j\neq k}\int_{0}^{\infty}
\frac{|k|^{r}}{|k-j|^{1+r}|j|^{r+\sigma-1}}|k-j|^{r}|\hat{u}(k-j,s)||j|^{r+\sigma}|\hat{v}(j,s)|\ \mathrm{d}s
\\
\leq \|u\|_{\mathcal{PM}^{r}}\|v\|_{\mathcal{Z}^{r+\sigma}}\left(\sup_{k\in\mathbb{Z}_{*}}\sum_{j\in\mathbb{Z}_{*},j\neq k}
\frac{|k|^{r}}{|k-j|^{1+r}|j|^{r+\sigma-1}}\right).
\end{multline}
We need this last quantity on the right-hand side to be finite; we postpone the proof of this to Lemma \ref{finalLemma}
below.  Given this lemma, we have completed our bound of $B_{adv}(u,v)$ in the space $\mathcal{PM}^{r}.$

Finally, we turn to the estimate for $B_{adv}(u,v)$ in the space $\mathcal{Z}^{r+\sigma}.$  Writing the norm of $B_{adv}(u,v)$
in this space, after some elementary considerations, we arrive at
\begin{multline}\nonumber
\|B_{adv}(u,v)\|_{\mathcal{Z}^{r+\sigma}} =
\\
\sup_{k\in\mathbb{Z}_{*}}\int_{0}^{\infty}\left| |k|^{r+\sigma}
\sum_{j\in\mathbb{Z}_{*},j\neq k}\int_{0}^{t}e^{-(t-s)(\nu|k|^{\sigma}+i\bar{\omega}\mathrm{sgn}(k))}
\frac{\hat{u}(k-j,s)}{|k-j|}ij\hat{v}(j,s)\ \mathrm{d}s\right| \mathrm{d}t
\\
\leq \sup_{k\in\mathbb{Z}_{*}}\sum_{j\in\mathbb{Z}_{*},j\neq k}|k|^{r+\sigma}\int_{0}^{\infty}\int_{0}^{t}
e^{-(t-s)\nu|k|^{\sigma}}\frac{|\hat{u}(k-j,s)|}{|k-j|}|j||\hat{v}(j,s)|\ \mathrm{d}s\mathrm{dt}.
\end{multline}
As before, we change the order of integration and carry out the integration with respect to $t.$  This gives
the bound
\begin{equation}\label{toBeSeenAgainAgain}
\|B_{adv}(u,v)\|_{\mathcal{Z}^{r+\sigma}} \leq \frac{1}{\nu}
\sup_{k\in\mathbb{Z}_{*}}\sum_{j\in\mathbb{Z}_{*},j\neq k}\int_{0}^{\infty}
|k|^{r}\frac{|\hat{u}(k-j,s)|}{|k-j|}|j||\hat{v}(j,s)|\ \mathrm{d}s.
\end{equation}
The right-hand side here is the same as the right-hand side in \eqref{toBeSeenAgain}, which we have already bounded.
This completes the proof.
\end{proof}

We now prove the final lemma.
\begin{lemma}\label{finalLemma}
Let $\sigma>1$ and $r\in\left(\frac{1-\sigma}{2},0\right]$ be given.  Then
\begin{equation}\label{finalLemmaQuantity}
\sup_{k\in\mathbb{Z}_{*}}\sum_{j\in\mathbb{Z}_{*},j\neq k}
\frac{|k|^{r}}{|k-j|^{1+r}|j|^{r+\sigma-1}} < \infty.
\end{equation}
\end{lemma}
\begin{proof}
Let $k\in\mathbb{Z}_{*}$ be given; without loss of generality, we may assume $k>0.$

We consider two cases.  First, if $r<-1,$ then there exists $c>0$ such that $|k-j|^{-(1+r)}\leq c(|k|^{-(1+r)}+|j|^{-(1+r)})$
(as proved in the appendix of \cite{ALSS}).
Then we have the estimate
\begin{equation*}
\sum_{j\in\mathbb{Z}_{*},j\neq k}\frac{|k|^{r}}{|k-j|^{1+r}|j|^{r+\sigma-1}}
\leq
c \sum_{j\in\mathbb{Z}_{*},j\neq k}\frac{|k|^{r}|k|^{-(1+r)}}{|j|^{r+\sigma-1}}
+
c \sum_{j\in\mathbb{Z}_{*},j\neq k}\frac{|k|^{r}|j|^{-(1+r)}}{|j|^{r+\sigma-1}}.
\end{equation*}
We simplify, and we estimate $|k|^{-1}\leq1,$ finding
\begin{equation*}
\sum_{j\in\mathbb{Z}_{*},j\neq k}\frac{|k|^{r}}{|k-j|^{1+r}|j|^{r+\sigma-1}}
\leq
c \sum_{j\in\mathbb{Z}_{*},j\neq k}\frac{1}{|j|^{r+\sigma-1}}
+
c \sum_{j\in\mathbb{Z}_{*},j\neq k}\frac{1}{|j|^{2r+\sigma}}.
\end{equation*}
Our conditions on $r$ imply $r+\sigma-1>-r>1,$ and also $2r+\sigma>1.$
Therefore these sums converge, and are clearly bounded independently of $k.$

We next consider $r\geq-1.$
We decompose the relevant sum as $I+II+III+IV,$ where
\begin{equation*}
I=\sum_{j=-\infty}^{-1}\frac{|k|^{r}}{|k-j|^{1+r}|j|^{r+\sigma-1}},\qquad II=\sum_{j=1}^{k/2}\frac{|k|^{r}}{|k-j|^{1+r}|j|^{r+\sigma-1}},
\end{equation*}
\begin{equation*}
III=\sum_{j=\frac{k}{2}+1}^{k-1}\frac{|k|^{r}}{|k-j|^{1+r}|j|^{r+\sigma-1}},\qquad
IV=\sum_{j=k+1}^{\infty}\frac{|k|^{r}}{|k-j|^{1+r}|j|^{r+\sigma-1}}.
\end{equation*}
For $I,$ we have $|k-j|\geq|j|.$  Keeping in mind $r\leq0$ and $\sigma+2r>1,$ we find
\begin{equation*}
I\leq\sum_{j\in\mathbb{Z}_{*}}\frac{1}{|j|^{\sigma+2r}},
\end{equation*}
and this sum converges and is clearly independent of $k.$
For $II,$ we have $|k-j|=k-j\geq \frac{k}{2}\geq j.$  Then we again have
\begin{equation*}
II\leq \sum_{j\in\mathbb{Z}_{*}}\frac{1}{|j|^{\sigma+2r}}.
\end{equation*}

For $III,$ we have $|k-j|=k-j,$ and $j>k-j,$ so $\frac{1}{j}<\frac{1}{k-j}.$  This implies the bound
\begin{equation*}
III=k^{r}\sum_{j=\frac{k}{2}+1}^{k-1}\frac{1}{(k-j)^{\sigma+2r}}.
\end{equation*}
Changing variables to $\ell=k-j,$ this becomes
\begin{equation*}
III=k^{r}\sum_{\ell=1}^{\frac{k}{2}-1}\frac{1}{\ell^{\sigma+2r}}\leq\sum_{\ell\in\mathbb{Z}_{*}}\frac{1}{|\ell|^{\sigma+2r}}.
\end{equation*}
Again, this converges and is independent of $k.$

For $IV,$ we use $|k-j|=j-k,$ and we change variables to $\ell=j-k.$
This gives the expression
\begin{equation*}
IV=k^{r}\sum_{\ell=1}^{\infty}\frac{1}{\ell^{1+r}(\ell+k)^{r+\sigma-1}}.
\end{equation*}
Since $k>0,$ we have the estimate $\ell+k>\ell,$ leading to the bound
\begin{equation*}
IV\leq k^{r}\sum_{\ell=1}^{\infty}\frac{1}{\ell^{\sigma+2r}}.
\end{equation*}
With our conditions on $r,$ this is again bounded independently of $k.$
This completes the proof of the lemma.
\end{proof}

\begin{remark} Whether the above argument does also give a result for $\sigma=1$ hinges on whether a version of Lemma
\ref{finalLemma} holds for $r>0.$   Clearly \eqref{finalLemmaQuantity} does not hold for $\sigma=1$ and $r=0.$
To verify \eqref{finalLemmaQuantity} for $\sigma=1$ and $r>0,$ we must consider sums of the form (with $k$ taken to be positive
without loss of generality)
\begin{equation*}
k^{r}\sum_{j=1}^{k-1}\frac{1}{(k-j)^{1+r}j^{r}}.
\end{equation*}
That this is bounded uniformly with respect to $k$ can be straightforwardly shown for some simple values of $r$ such
as $r=1$ or $r=1/2.$  It appears to also be true for general $r>0,$
but fully exploring this is beyond the scope of the
present work.
\end{remark}

\subsection{Analyticity of the solutions}\label{analyticitySection}

The solutions of the generalized Constantin-Lax-Majda equation which we have shown to exist are in fact
analytic at all positive times.  This is the content of the next theorem.

\begin{theorem}\label{analyticityTheorem} Let $a \in \mathbb{R}$. Assume either:

\vspace{.1cm}

\begin{enumerate}
\item[(H1)] $\sigma\geq 1$ is given, or
\item[(H2)] $\sigma>1$ and $r\in\left(\frac{1-\sigma}{2},0\right]$ are given.
\end{enumerate}

\vspace{.3cm}

Let $\alpha \in (0,1)$. Then there exists $\varepsilon> 0$ such that,

\vspace{.1cm}

\begin{enumerate}

\item if $\mathrm{(H1)}$  holds then, for any $\bar{\omega}\in\mathbb{R}$ and $\omega_{0}\in Y^{-\sigma/2}$
satisfying $\|\omega_{0}\|_{Y^{-\sigma/2}}<\varepsilon,$ the solution $\tilde{\omega}=\bar{\omega}+\omega,$
with initial data $\tilde{\omega}_{0}=\bar{\omega}+\omega_{0},$ and with
$\omega\in\mathcal{Y}^{-\sigma/2}\cap\mathcal{X}^{\sigma/2}$, of \eqref{gCLMd}
given by Theorem \ref{CLMAdvectionXYTheorem} is analytic, and,

\item if $\mathrm{(H2)}$ holds then, for any $\bar{\omega}\in\mathbb{R}$ and $\omega_{0}\in PM^{r}$
satisfying $\|\omega_{0}\|_{PM^{r}}<\varepsilon,$ the solution $\tilde{\omega}=\bar{\omega}+\omega,$ with initial data $\tilde{\omega}_{0}=\bar{\omega}+\omega_{0},$ and with  $\omega\in\mathcal{PM}^{r}\cap\mathcal{Z}^{r+\sigma}$, of \eqref{gCLMd} given by Theorem \ref{ExistsAdvPMr} is analytic.
\end{enumerate}

In both cases the radius of analyticity $R_\alpha$ satisfies $R_\alpha \geq \max\{\nu t^{1/\sigma} , \alpha \nu t  \}$.
\end{theorem}

%\begin{theorem}\label{analyticityTheorem}
%\begin{enumerate}
%\item Let $a\in\mathbb{R}$ and
%$\sigma\geq1$  be given.
%Let $\alpha \in (0,1)$.
%Then there exists $\varepsilon> 0$ such that,
%for any $\bar{\omega}\in\mathbb{R}$ and $\omega_{0}\in Y^{-\sigma/2}$
%satisfying $\|\omega_{0}\|_{Y^{-\sigma/2}}<\varepsilon,$ the solution $\tilde{\omega}=\bar{\omega}+\omega,$
%with initial data $\tilde{\omega}_{0}=\bar{\omega}+\omega_{0},$ and with
%$\omega\in\mathcal{Y}^{-\sigma/2}\cap\mathcal{X}^{\sigma/2}$, of \eqref{gCLMd}
%given by Theorem \ref{CLMAdvectionXYTheorem} is analytic and has radius of analyticity $R_\alpha \geq \max\{\nu t^{1/\sigma} %, \alpha \nu t  \}$.

%\item Let $a\in\mathbb{R},$
%$\sigma>1$ and $r\in\left(\frac{1-\sigma}{2},0\right]$ be given.
%Let $\alpha \in (0,1)$.
%Then there exists $\varepsilon> 0$ such that,
%for any $\bar{\omega}\in\mathbb{R}$ and $\omega_{0}\in PM^{r}$
%satisfying $\|\omega_{0}\|_{PM^{r}}<\varepsilon,$ the solution $\tilde{\omega}=\bar{\omega}+\omega,$ with initial data %$\tilde{\omega}_{0}=\bar{\omega}+\omega_{0},$ and with  $\omega\in\mathcal{PM}^{r}\cap\mathcal{Z}^{r+\sigma}$, of %\eqref{gCLMd}
%given by Theorem \ref{ExistsAdvPMr} is analytic and has radius of analyticity $R_\alpha \geq \max\{\nu t^{1/\sigma} , \alpha %\nu t  \}$.
%\end{enumerate}
%\end{theorem}

The rest of this section will be devoted to the proof of Theorem \ref{analyticityTheorem}.  We first need two auxiliary claims, which we now state. Let $\sigma \geq 1$ and $\alpha \in (0,1)$.

\begin{claim} \label{Claim1}
\begin{itemize}
  \item[]
  \item There exists a constant $C_1>0$ such that, if $b(t) = t^{1/\sigma}$, then \[e^{\nu t^{1/\sigma} |k|}e^{ - \nu t|k|^\sigma} \leq C_1e^{-\nu t|k|^\sigma/2},\] for all $t>0$ and all $k \in \mathbb{Z}_\ast$.
  \item If $b(t) = \alpha t$ then \[e^{\nu \alpha t |k|}e^{ - \nu t|k|^\sigma} \leq e^{ - (1-\alpha)\nu t|k|^\sigma},\] for all $t>0$ and all $k \in \mathbb{Z}_\ast$.
\end{itemize}

\end{claim}

\begin{claim} \label{Claim2}

\begin{itemize}
  \item[]
  \item There exists a constant $C_2>0$ such that, if $b(t) = t^{1/\sigma}$ then
  \[e^{\nu t^{1/\sigma} |k|}e^{ - \nu s^{1/\sigma} |k|}e^{ - \nu (t-s)|k|^\sigma} \leq C_2e^{-\nu (t-s)|k|^\sigma/2},\]
  for all $0<s<t$ and all $k \in \mathbb{Z}_\ast$.
  \item If $b(t) = \alpha t$ then \[e^{\nu \alpha t |k|}e^{ - \nu \alpha s |k|}e^{- \nu (t-s)|k|^\sigma} \leq e^{ - (1-\alpha)\nu (t-s)|k|^\sigma},\] for all  $0<s<t$ and all $k \in \mathbb{Z}_\ast$.
\end{itemize}

\end{claim}

The proof of Claim \ref{Claim1} is a slight adaptation of the proof of \cite[Claim 1]{ALN6}, given in
\cite[(20) and (23)]{ALN4}. Similarly, the proof of Claim \ref{Claim2} is a slight adaptation of the proof of \cite[Claim 2]{ALN6}, given in \cite[(21) and (25)]{ALN4}. We therefore omit the proof of Claims \ref{Claim1} and \ref{Claim2}.

\begin{proof}[Proof of Theorem \ref{analyticityTheorem}]  The proof of items (1) and (2) of the theorem are very similar, and we only prove (2).
The result is obtained by proving that $e^{\nu b(t)\Lambda}\omega \in \mathcal{PM}^{r}\cap\mathcal{Z}^{r+\sigma}$  in both cases: $b(t) = t^{1/\sigma}$ and $b(t) = \alpha t$, $0<\alpha <1$.

Now, $\omega$ is a mild solution of \eqref{gCLMd2}, so we may write
\begin{equation}\label{e:expbDw}
 e^{\nu b(t)\Lambda}\omega =  e^{\nu b(t)\Lambda}e^{-t(\nu \Lambda^{\sigma} - \bar{\omega}H)}\omega_0 + e^{\nu b(t)\Lambda}\mathcal{B}(\omega,\omega),
\end{equation}
where $\mathcal{B}=B+aB_{adv}$ was introduced in \eqref{e:calBil}. Let us set
\[W(t):= e^{\nu b(t)\Lambda}\omega\]
and let us express \eqref{e:expbDw} in terms of $W$, so that:
\begin{equation} \label{e:Weqn}
W(t) = e^{\nu b(t)\Lambda}e^{-t(\nu \Lambda^\sigma - \bar{\omega}H)}\omega_0 + e^{\nu b(t)\Lambda}\mathcal{B}(e^{-\nu b(t)\Lambda}W,e^{-\nu b(t)\Lambda}W),
\end{equation}
where we used $W(0)=\omega_0$. We will show that, if $\omega_0$ is sufficiently small in $PM^r$, then there exists a globally-defined $W \in \mathcal{PM}^{r}\cap\mathcal{Z}^{r+\sigma}$ which satisfies \eqref{e:Weqn}. To this end we must show that the linear term is continuous from $PM^r$ to $\mathcal{PM}^{r}\cap\mathcal{Z}^{r+\sigma}$ and that the bilinear term is continuous from $\mathcal{PM}^{r}\cap\mathcal{Z}^{r+\sigma} \times \mathcal{PM}^{r}\cap\mathcal{Z}^{r+\sigma}$ to $\mathcal{PM}^{r}\cap\mathcal{Z}^{r+\sigma}$.

We first examine the linear term, noting that
\[|(e^{\nu b(t)\Lambda}e^{-t(\nu \Lambda^\sigma - \bar{\omega}H)}\omega_0)^\wedge(k,t)| = e^{\nu b(t) |k| - \nu t|k|^\sigma}|\hat{\omega_0}(k)|.
\]

Then, using Claim \ref{Claim1} we can reproduce the same reasoning as in \eqref{e:lintermgCLMd2PMr} and \eqref{e:lintermgCLMd2Zrandsigma} to deduce that, in both cases $b(t)= t^{1/\sigma}$ and $b(t)=\alpha t$,
\begin{equation}\label{e:lintermgCLMd2ANPMr}
  \| e^{\nu b(t)\Lambda}e^{-t(\nu \Lambda - \bar{\omega}H)}\omega_0 \|_{\mathcal{PM}^{r}} \leq C  \| \omega_0 \|_{PM^r},
\end{equation}
and
\begin{equation}\label{e:lintermgCLMd2ANZrandsigma}
 \|e^{\nu b(t)\Lambda}e^{-t(\nu \Lambda - \bar{\omega}H)}\omega_0\|_{\mathcal{Z}^{r+\sigma}} \leq \frac{ C}{\nu} \| \omega_0 \|_{PM^r} .
\end{equation}
This establishes the continuity of the linear term from $PM^r$ to $\mathcal{PM}^{r}\cap\mathcal{Z}^{r+\sigma}$.

Next we need to estimate the bilinear term:
\begin{align*}
e^{\nu b(t)\Lambda}\mathcal{B}(e^{-\nu b(t)\Lambda}U,e^{-\nu b(t)\Lambda}V) & = e^{\nu b(t)\Lambda}B(e^{-\nu b(t)\Lambda}U,e^{-\nu b(t)\Lambda}V) \\
& + e^{\nu b(t)\Lambda} B_{adv}(e^{-\nu b(t)\Lambda}U,e^{-\nu b(t)\Lambda}V).
\end{align*}

Let us analyze each of the two right-hand-side terms separately. First note that, from the definition of $B(u,v)$ in \eqref{CLMd2Bilinear} and from the identity \eqref{CLMd2BilinearYest} together with the definition of the $\mathcal{Y}^s$-norm it follows that
\begin{align*}
(e^{\nu b(t)\Lambda}& B(e^{-\nu b(t)\Lambda}U,e^{-\nu b(t)\Lambda}V))^{\wedge}(k,t)
=
\int_{0}^{t}
e^{\nu b(t) |k|} e^{-(t-s)(\nu|k|^{\sigma}+i\bar{\omega}\mathrm{sgn}(k))} \\
& \left( \sum_{j\in\mathbb{Z}_{*},j\neq k} e^{-\nu b(s) |k-j|}\hat{U}(k-j,s)
(-i\mathrm{sgn}(j))e^{-\nu b(s) |j|}\hat{V}(j,s)\right)
\ \mathrm{d}s.
\end{align*}
Therefore, from the triangle inequality $-|k-j|-|j|\leq -|k|$ we conclude that
\begin{align*}
 &\left| (e^{\nu b(t)\Lambda} B(e^{-\nu b(t)\Lambda}U,e^{-\nu b(t)\Lambda}V))^{\wedge}(k,t)\right|\\
&\leq \int_{0}^{t} e^{\nu b(t)|k|}e^{-\nu b(s) |k|} e^{-\nu(t-s)|k|^{\sigma}} \left(
\sum_{j\in\mathbb{Z}_{*},j\neq k} | \hat{U}(k-j,s)| |\hat{V}(j,s)| \right) \ \mathrm{d}s.
\end{align*}
Now, using the estimates in Claim \ref{Claim2} and applying the same reasoning as in \eqref{CLMTermPMr} and \eqref{CLMtermZrandsigmaEST} we find, for both $b(t)=t^{1/\sigma}$ and $b(t)=\alpha t$, that
\begin{align}\label{BiltermgCLMd2AN}
&\| e^{\nu b(t)\Lambda} B(e^{-\nu b(t)\Lambda}U,e^{-\nu b(t)\Lambda}V) \|_{\mathcal{PM}^{r}} + \| e^{\nu b(t)\Lambda} B(e^{-\nu b(t)\Lambda}U,e^{-\nu b(t)\Lambda}V) \|_{\mathcal{Z}^{r+\sigma}}\nonumber \\
&\leq C \sup_{k\in\mathbb{Z}_{*}}\sum_{j\in\mathbb{Z}_{*},j\neq k}\int_{0}^{\infty} |k|^{r} |\hat{U}(k-j,s)|
|\hat{V}(k-j,s)|\ \mathrm{d}s,
\end{align}
for some constant $C=C(\nu,\alpha)$. As noted in the proof of Theorem \ref{ExistsAdvPMr}, the right-hand-side above may be estimated by $C\|U\|_{\mathcal{PM}^r}\| V \|_{\mathcal{Z}^{r+\sigma}}$.

Lastly, we address the $B_{adv}$ term. From the definition of $B_{adv}$ given in \eqref{e:advBil} and the identity in \eqref{toBeSeenAgain} it follows that
\begin{align*}
(e^{\nu b(t)\Lambda}& B_{adv}(e^{-\nu b(t)\Lambda}U,e^{-\nu b(t)\Lambda}V))^{\wedge}(k,t)
=
\int_{0}^{t}
e^{\nu b(t) |k|} e^{-(t-s)(\nu|k|^{\sigma}+i\bar{\omega}\mathrm{sgn}(k))} \\
& \left( \sum_{j\in\mathbb{Z}_{*},j\neq k} e^{-\nu b(s) |k-j|}\frac{\hat{U}(k-j,s)}{|k-j|}
e^{-\nu b(s) |j|}ij\hat{V}(j,s)\right)
\ \mathrm{d}s.
\end{align*}
Hence, from the triangle inequality we obtain
\begin{align*}
 &\left| (e^{\nu b(t)\Lambda} B_{adv}(e^{-\nu b(t)\Lambda}U,e^{-\nu b(t)\Lambda}V))^{\wedge}(k,t)\right|\\
&\leq \int_{0}^{t} e^{\nu b(t)|k|}e^{-\nu b(s) |k|} e^{-\nu(t-s)|k|^{\sigma}} \left(
\sum_{j\in\mathbb{Z}_{*},j\neq k}  \frac{|\hat{U}(k-j,s)|}{|k-j|} |j||\hat{V}(j,s)| \right) \ \mathrm{d}s.
\end{align*}
Using again the estimates in Claim \ref{Claim2} and applying the same reasoning as in \eqref{toBeSeenAgain} and in the estimates leading up to \eqref{toBeSeenAgainAgain} we find, for both $b(t)=t^{1/\sigma}$ and $b(t)=\alpha t$, that
\begin{align}\label{BiladvtermgCLMd2AN}
& \| e^{\nu b(t)\Lambda} B_{adv}(e^{-\nu b(t)\Lambda}U,e^{-\nu b(t)\Lambda}V) \|_{\mathcal{PM}^{r}} +
\| e^{\nu b(t)\Lambda}  B_{adv}(e^{-\nu b(t)\Lambda}U,e^{-\nu b(t)\Lambda}V) \|_{\mathcal{Z}^{r+\sigma}} \nonumber \\
& \leq C \sup_{k\in\mathbb{Z}_{*}}\sum_{j\in\mathbb{Z}_{*},j\neq k}\int_{0}^{\infty} |k|^{r} \frac{|\hat{U}(k-j,s)|}{|k-j|} |j||\hat{V}(k-j,s)|\ \mathrm{d}s,
\end{align}
for some constant $C=C(\nu,\alpha)$. Once again recalling the proof of Theorem \ref{ExistsAdvPMr}, the right-hand-side above may be estimated by $C\|U\|_{\mathcal{PM}^r}\| V \|_{\mathcal{Z}^{r+\sigma}}$.

It follows from \eqref{BiltermgCLMd2AN} and \eqref{BiladvtermgCLMd2AN} that the bilinear term is continuous from $\mathcal{PM}^{r}\cap\mathcal{Z}^{r+\sigma} \times \mathcal{PM}^{r}\cap\mathcal{Z}^{r+\sigma}$ to $\mathcal{PM}^{r}\cap\mathcal{Z}^{r+\sigma}$. We may now apply the abstract fixed point result Lemma \ref{fixedpoint} to obtain $e^{b(t)\Lambda} \omega \equiv W \in \mathcal{PM}^{r}\cap\mathcal{Z}^{r+\sigma}$ as long as $\omega_0$ is sufficiently small in $PM^r$.

Since $W\in\mathcal{PM}^{r},$ for any $t$ we have that $W(\cdot,t)\in PM^{r}.$  Thus, we have the bound
\begin{equation*}
e^{\nu b(t)|k|}|\hat{\omega}(k,t)||k|^{r}\leq C_{1},
\end{equation*}
for either choice of $b(t),$ for all $k$ and $t,$ and for $C_{1}=\|W\|_{\mathcal{PM}^{r}}.$  Rearranging this, we get
\begin{equation*}
|\hat{\omega}(k,t)|\leq C_{1}|k|^{-r}e^{-\nu b(t)|k|}\leq C_{2}e^{-\beta|k|},
\end{equation*}
for any $\beta<\nu b(t),$ for some $C_{2}.$
As in Theorem IX.13 of \cite{reedSimon}, this exponential decay rate of the Fourier coefficients
implies that $\omega$ is analytic with radius of analyticity, $R_\alpha,$ with $R_{\alpha}$
satisfying $R_{\alpha}\geq\nu b(t).$  Consequently, the same is true for $\tilde{\omega}$.  This completes the proof of the theorem.
\end{proof}

\section{Discussion}\label{discussionSection}

We now close with some comments.  We have presented two methods for proving the existence and analyticity of solutions
of parabolic equations with low-regularity data, and we have illustrated both methods on problems of current interest.  For
generalized one-dimensional Kuramoto-Sivashinsky equations, we have clarified a point raised by Papageorgiou, Smyrlis, and
Tomlin on the question of whether solutions of \eqref{ottoFamily} are analytic at positive times in a borderline case.  For
generalizations of the Constantin-Lax-Majda equation, we have established global existence of solutions with quite rough data.

The data we may take using the second method is much rougher than the data we may take in the first method.  The advantage
of the first method is that more general nonlinearities may be treated.  While we showed this with a general power nonlinearity
for the generalized one-dimensional Kuramoto-Sivashinsky equation, in other applications such a general nonlinearity arises
in a physically relevant way.  For example the first author in \cite{BLMS} used this method to establish global existence of solutions
for a problem in epitaxial growth of thin films; there the nonlinearity was analytic, i.e. had an infinite power series expansion
(essentially $e^{x}-1-x$).
Indeed, this first method is based on the work of Duchon and Robert for the vortex sheet \cite{duchonRobert}, in which they
used analyticity of the Birkhoff-Rott integral to express it with an infinite power series expansion.
For such problems, our second method does not apply as it requires the nonlinearity be only quadratic.

For the generalized Constantin-Lax-Majda problem, with diffusion parameter $\sigma>1,$ we have used this second method to
prove existence of solutions
with data in the space $Y^{-\sigma/2}$ and with data in the space $PM^{r}$ for $\frac{1-\sigma}{2}<r\leq 0.$  We show now that
these spaces are not comparable (neither is a subset of the other).
Clearly $PM^{r}$ is
not contained in $Y^{-\sigma/2},$ as we could take $\hat{\omega}_{0}(k)=1$ for $k\neq0$ (and $\hat{\omega}_{0}(0)=0$).  Then
$\omega_{0}\in PM^{0}\subseteq PM^{r},$ but $\omega_{0}\notin Y^{-\sigma/2}$ at least for $\sigma\leq 2.$  (Of course, $0<\sigma\leq2$
is the physical range for dissipation.)  Alternatively, let $\varepsilon\in(0,1/2),$ and let $\omega_{0}$  be such that
\[
|\hat{\omega}_{0}(j)| =
\left\{
\begin{array}{ll}
|j|^{\frac{\sigma}{2}-\varepsilon} & \text{ if } |j|= 2^{\ell}, \; \ell = 1, 2, \ldots \\
0 & \text{ otherwise. }
\end{array}
\right.
\]
Then $\omega_{0}\in Y^{-\sigma/2}$ but $\omega_{0}\notin PM^{(1-\sigma)/2},$
and thus $\omega_{0}\notin PM^{r}.$

We now discuss critical spaces for the dissipative Constantin-Lax-Majda equation, with and without advection.  The Constantin-Lax-Majda
equation is meant as a one-dimensional model of vortex stretching for the three-dimensional Euler or Navier-Stokes equations.
For the three-dimensional Navier-Stokes equations, Lei and Lin have shown existence of global solutions with small data in $Y^{-1},$
and Cannone and Karch have shown existence of global solutions with small data in $PM^{2}$  \cite{cannoneKarch}, \cite{leiLin}.  These
are both critical spaces for the three-dimensional Navier-Stokes equations.

The CLM equation \eqref{gCLMd} has a natural scaling. Indeed, if $v=v(t,x)$ is a (strong) solution of \eqref{gCLMd}, then, for any $\lambda > 0$, $u(t,x) \equiv \lambda v(\lambda t, \lambda^{1/\sigma} x)$ is also a solution. Criticality, within a given scale of spaces, means the specific regularity parameter for which the spatial norm of initial data does not depend on $\lambda$. Regularity above this threshold is called subcritical, and below subcritical. Not every equation has a natural scaling. For example, the Kuramoto-Sivashinsky equation \eqref{mainEquation}, with $\nu>0$ does not have a natural scaling, but many physically relevant models possess a natural scaling, such as the Navier-Stokes system or the KdV equation.

For the spaces $Y^s$, it is easy to see that $Y^{-\sigma}$ is critical for \eqref{gCLMd}. This means that the existence result obtained here, for initial data in $Y^{-\sigma/2}$, despite improving previously known results, is still subcritical. For $PM^s$ criticality means $s=1-\sigma$, and our existence result in $PM^{(1-\sigma)/2}$ is subcritical as well.

\section*{Acknowledgments}
DMA is grateful to the National Science Foundation for support through grant
DMS-2307638.  DMA also thanks Jonathan Goodman for the suggestion of writing the first method
of the present paper in the form of an abstract existence result. MCLF was partially supported
by CNPq, through grant \# 304990/2022-1, and by FAPERJ, through  grant \# E-26/201.209/2021.
HJNL acknowledges the support of CNPq, through  grant \# 305309/2022-6, and that of FAPERJ,
through  grant \# E-26/201.027/2022.

\bibliographystyle{plain}
\bibliography{Otto-CLM.bib}

\end{document}